\documentclass[10pt,reqno]{amsart}
\usepackage{amssymb,amsfonts,amsthm,amscd,stmaryrd,dsfont,esint,upgreek,constants,todonotes}
\usepackage[shortlabels]{enumitem}
\usepackage{mathtools, mathrsfs, xargs}
\usepackage{hyperref}
\hypersetup{colorlinks, citecolor=violet, linkcolor=blue, urlcolor=Blue}
\newif\iflabels
\labelsfalse % change to \labelstrue to display labels in the margin
\iflabels 
  \usepackage[left=0.5in, right=2.5in, marginparwidth=2.2in]{geometry}
  \usepackage{showlabels}
\else
  \usepackage[left=1.25in, right=1.25in]{geometry}
\fi

\numberwithin{equation}{section}
\theoremstyle{plain}                % title and number in bold, text italics
\newtheorem{theorem}{Theorem}[section]
\newtheorem{lemma}[theorem]{Lemma}
\newtheorem{proposition}[theorem]{Proposition}
\newtheorem{corollary}[theorem]{Corollary}

\theoremstyle{definition}           % title and number in bold, text normal
\newtheorem{definition}[theorem]{Definition}
\newtheorem{example}[theorem]{Example}

\theoremstyle{remark}
\newtheorem{remark}[theorem]{Remark}

\newcommand{\define}[1]{{\textbf{#1}}}

\providecommand{\alias}{}
\renewcommand{\alias}[1]{\providecommand{#1}{}\renewcommand{#1}}

  \DeclarePairedDelimiter\ab{\langle}{\rangle} % <angle brackets>
  \DeclarePairedDelimiter\abs{\lvert}{\rvert}   % |absolute value|
  \DeclarePairedDelimiter\norm{\lVert}{\rVert}  % ||norm||
  \DeclarePairedDelimiterX\set[1]\{\}{ #1 }
  \DeclarePairedDelimiterX\sets[2]\{\}{ #1\,:\,#2 }

  \let\bPeexp\exp
  \let\exp\relax
  \DeclarePairedDelimiterXPP\exp[1]{\bPeexp}(){}{#1}

  \makeatletter
    \let\oldabs\abs \def\abs{\@ifstar{\oldabs}{\oldabs*}}
    \let\oldab\ab \def\ab{\@ifstar{\oldab}{\oldab*}}
    \let\oldnorm\norm \def\norm{\@ifstar{\oldnorm}{\oldnorm*}}
    \let\oldexp\exp \def\exp{\@ifstar{\oldexp}{\oldexp*}}
  \makeatother

\DeclareMathOperator*\esssup{esssup}

\makeatletter
  \newcommand{\opnorm}{\@ifstar\@opnorms\@opnorm}
  
  \newcommand{\@opnorm}[2][]{%
    \mathopen{#1|\mkern-1.5mu#1|\mkern-1.5mu#1|}
    #2
    \mathclose{#1|\mkern-1.5mu#1|\mkern-1.5mu#1|}
  }
\makeatother

\alias{\R}{{\mathbb R}}
\alias{\C}{{\mathbb C}}
\alias{\Z}{{\mathbb Z}}
\alias{\N}{{\mathbb N}}

\newcommand{\pd}[2]{\frac{\partial #1}{\partial #2}}

\newcommand{\Implies}{\Rightarrow}

\newcommand{\sP}{\mathcal{P}}
\newcommand{\sS}{\mathcal{S}}

\newcommand{\sF}{\mathcal{F}}
\newcommand{\sG}{\mathcal{G}}

\newcommand{\bP}{\mathbb{P}}
\newcommand{\bF}{\mathbb{F}}
\newcommand{\bE}{\mathbb{E}}
\newcommand{\bQ}{\mathbb{Q}}
\newcommand{\bD}{\mathbb{D}}

\newcommand{\klambda}{\boldsymbol{\lambda}}
\newcommand{\ka}{\boldsymbol{a}}

\newcommand{\kz}{\boldsymbol{z}}
\newcommand{\kA}{\boldsymbol{A}}
\newcommand{\kB}{\boldsymbol{B}}

\newcommand{\kV}{\boldsymbol{V}}
\newcommand{\kZ}{\boldsymbol{Z}}

\newcommand{\kb}{\boldsymbol{b}}

\newcommand{\lone}{L^1}

\newcommand{\lpee}{L^p}
\newcommand{\lque}{L^q}
\newcommand{\linf}{L^{\infty}}
\newcommand{\lques}{L^{q^*}}

\newcommand{\loi}{L^{1,\infty}}

\newcommand{\loq}{L^{1,q}}

\newcommand{\lto}{L^{2,1}}
\newcommand{\ltt}{L^{2,2}}
\newcommand{\ltp}{L^{2,p}}
\newcommand{\ltq}{L^{2,q}}
\newcommand{\ltqs}{L^{2.q^*}}

\newcommand{\lii}{L^{\infty, \infty}}

\newcommand{\sinf}{\sS^{\infty}}

\newcommand{\spee}{\sS^p}
\newcommand{\sque}{\sS^q}
\newcommand{\sques}{\sS^{q^*}}

\newcommand{\BMO}{\text{BMO}}
\newcommand{\bmo}{\text{bmo}}
\newcommand{\bmoh}{\bmo^{1/2}}

\newcommand{\mpee}{\mathcal{M}^p}
\newcommand{\m}{\mathcal{M}}

\newcommand{\rp}{(\text{R}_p)}

\newcommand{\rpc}{\text{R}_p}

\newcommand{\ro}{(\text{R}_1)}
\newcommand{\roc}{\text{R}_1}

\newcommand{\doi}{\bD^{1,\infty}}

\newcommand{\const}[1]{C=C(#1)}
\newcommand{\leqc}{\leq_C}

\newcommand{\lip}{\mathbf{Lip}}
\newcommand{\ql}{\mathbf{QL}}
\newcommand{\ud}{\mathbf{U}}

\newcommand{\apb}{\mathbf{A}}

\newcommand{\mal}{\mathbf{M}}

\newcommand{\sam}{\set{a_m}}

\begin{document}

\title[Reverse H\"older Inequality and Quadratic BSDEs]{The Reverse H\"older Inequality for Matrix-valued Stochastic Exponentials and Applications to Quadratic BSDE Systems}
\author{Joe Jackson} 
\address{Department of Mathematics, The University of Texas at Austin}
\email{jjackso1@utexas.edu}
%\dedicatory{Version: \today}
\begin{abstract}
In this paper, we study the connections between three concepts - the reverse H\"older inequality for matrix-valued martingales, the well-posedness of linear BSDEs with unbounded coefficients, and the well-posedness of quadratic BSDE systems. In particular, we show that a linear BSDE with bmo (bounded mean oscillation) coefficients is well-posed if and only if the stochastic exponential of a related matrix-valued martingale satisfies a reverse H\"older inequality. Furthermore, we give structural conditions under which these two equivalent conditions are satisfied. Finally, we apply our results on linear equations to obtain global well-posedness results for two new classes of non-Markovian quadratic BSDE systems with special structure. 
\end{abstract}
\thanks{During the preparation of this work the author has been supported by the National Science Foundation under Grant No. DGE1610403 (2020-2023). Any opinions, findings and conclusions or recommendations expressed in this material are those of the author(s) and do not necessarily reflect the views of the National Science Foundation (NSF).
}
\maketitle

\section{Introduction}

Fix a probability space $(\Omega, \sF, \bP)$ which hosts a $d$-dimensional Brownian motion $\kB$. If $M \in \BMO(\R^{n \times n})$, i.e. $M$ is an $n \times n$ matrix-valued process each of whose components is a BMO martingale (with respect to the filtration of $\kB$), then there is a (potentially unbounded) process $\kA$ taking values in $(\R^d)^{n \times n}$ such that $M = \int \kA d\kB$ (we refer to the notations section for our conventions regarding multidimensional processes, but the meaning here should be clear at least when $d = 1$). By analogy with the scalar case, we can associate to $M$ a stochastic exponential $S$, which is the unique solution to the matrix SDE 
\begin{align*}
    \begin{cases}
    dS = S \kA d\kB, \\
    S_0 = I_{n \times n}.
    \end{cases}
\end{align*}
The first goal of this paper is to study the relationship between the well-posedness of the linear BSDE
\begin{align} \label{lin-bsde-1}
Y = \xi + \int_{\cdot}^T \big(\kA \kZ + \beta \big)  dt - \int_{\cdot}^T \kZ d \kB
\end{align}
and the properties of the stochastic exponential $S$. When $n = 1$, it is well known (see \cite{Kazamaki}) that $S$ is a uniformly integrable martingale which satisfies the reverse H\"older inequality $\rp$, i.e. the inequality 
\begin{align} \label{rhintro}
    \bE_{\tau}[|S_{\tau}^{-1}S_T|^p] \leq C
\end{align}
holds for some $p > 1$ and all stopping times $\tau$. This fact can be used to show that the BSDE \eqref{lin-bsde-1} is well-posed in an appropriate sense. When $n > 1$, $S$ need not even be a true martingale, as is demonstrated by counterexamples in \cite{Emery} and \cite{jackson2021existence}. In \cite{Delbaen-Tang}, Delbaen and Tang pointed out that a connection between the well-posedness of \eqref{lin-bsde-2} and the reverse H\"older inequality for $S$ still holds in higher dimensions, and part of the motivation for the present paper is to extend the results in Section 3 of \cite{Delbaen-Tang}.

Another goal of the present paper is to further explore the connection between the reverse H\"older inequality, linear BSDEs with $\bmo$ coefficients, and quadratic BSDE systems of the form
\begin{align} \label{introbsde}
    Y_{\cdot} = \xi + \int_{\cdot}^T f(\cdot, Y, \kZ) dt - \int_{\cdot}^T \kZ d\kB.
\end{align}
Here $f$ is a driver $f = f(t, \omega, y, \kz) : [0,T] \times \Omega \times \R^n \times (\R^d)^n \to \R^n$. This connection was first exploited in \cite{Briand-Elie}, where a-priori estimates for the linear equation \eqref{lin-bsde-1} were combined with a Malliavin calculus argument to produce an elegant proof of Kobylanski's original existence result (see \cite{Kobylanski}) for one-dimensional quadratic BSDEs. The approach used in \cite{Briand-Elie} fails in higher dimensions precisely because when $n > 1$, \eqref{lin-bsde-1} need not be well-posed in any reasonable sense. Nevertheless, it \textit{is} possible to adapt this approach to systems under smallness conditions (see \cite{harter2019}) or structural conditions (see \cite{jackson2021existence}).

\subsection{Our results}

\subsubsection{The relationship between S and BSDE(A)}
There are three main contributions of the paper. The first is to study the relationship between the exponential process $S$ and the well-posedness of \eqref{lin-bsde-1}. It turns out that uniqueness for \eqref{lin-bsde-1} is related to whether $S$ is a sufficiently integrable true martingale (see Propositions \ref{pro: yrep1} and \ref{pro:yrep}), while existence is related to the reverse H\"older inequality (see Propositions \ref{pro:linfexisth}, \ref{pro:linfexist} \ref{pro:lqexisth}, \ref{pro:lqexist}). Furthermore, if the equation is well-posed in $\sque$, then $S$ must be a true martingale satisfying $\rp$, where $q$ is the conjugate of $p$. This results in a simple equivalence which is stated formally in Theorem \ref{thm:main}. We summarize our results on the relationship between $S$ and \eqref{lin-bsde-1} in a table below. Throughout, $q$ denotes the conjugate of $p \in [1,\infty)$. 

\begin{table}[ht]
  \begin{center}
   \label{tab:table1}
   \begin{tabular}{l|l|l} % <-- Alignments: 1st column left, 2nd middle and 3rd right, with vertical lines in between
      \textbf{Assumption} & \textbf{Implication} & \textbf{Precise Statement}\\
      \hline
      \scriptsize{$S$ is a true martingale} & \scriptsize{Uniqueness in $\sinf$ holds for \eqref{lin-bsde-1}} & \scriptsize{Proposition \ref{pro: yrep1}}\\
      \scriptsize{$S$ is a true martingale and $S_T \in \lpee$, $p > 1$} & \scriptsize{Uniqueness in $\sque$ holds for $\eqref{lin-bsde-1}$} & \scriptsize{Proposition \ref{pro:yrep}}\\
      \scriptsize{$S$ satisfies $(R_1)$} & \scriptsize{Existence in $\sinf$ holds for \eqref{lin-bsde-1}} & \scriptsize{Propositions \ref{pro:linfexisth} and \ref{pro:linfexist}} \\
      \scriptsize{$S$ satisfies $(R_{p^*})$ for some $p^* > p$} & \scriptsize{Existence in $\sque$ holds for \eqref{lin-bsde-1}} & \scriptsize{Propositions \ref{pro:lqexisth} and \ref{pro:lqexist}} \\
      \scriptsize{\eqref{lin-bsde-1} is well-posed in $\sque$} & \scriptsize{$S$ is a true martingale satisfying $\rp$} & \scriptsize{Proposition \ref{pro:rpfromeqn}} 
    \end{tabular}
  \end{center}
\end{table}

We note that the basic connection between \eqref{lin-bsde-1} and the reverse H\"older inequality is present in \cite{Delbaen-Tang} and \cite{harter2019}, and our existence results are similar in spirit to Theorem 3.2 of \cite{Delbaen-Tang}. Nevertheless, because of some subtle issues in Section 3 of \cite{Delbaen-Tang} (see Remark \ref{rmk:errors}) and because we need existence results which are well-tailored for our applications to quadratic BSDEs, we give a full treatment of existence via reverse H\"older. The sufficient conditions for uniqueness and the equivalence provided by Theorem \ref{thm:main} are, to the best of the author's knowledge, new, and complement nicely the existence results from \cite{Delbaen-Tang}.
We note also that we are able to use a result of \cite{jackson2021existence} to extend most of the results obtained here to the more general linear BSDE
\begin{align*}
    Y_{\cdot} = \xi + \int_{\cdot}^T \big(\alpha Y + \kA \kZ + \gamma\big) dt - \int_{\cdot}^T \kZ d\kB,
\end{align*}
see Proposition \ref{pro:sliceable}.

\subsubsection{Structural conditions which guarantee reverse H\"older}

The second main contribution is to identify structural conditions on the matrix $\kA$ which guarantee that the equivalent conditions of Theorem \ref{thm:main} are satisfied. One of these structural conditions, namely triangularity, was studied already in \cite{jackson2021existence} (although only the BSDE \eqref{lin-bsde-1} was considered there, and no properties of $S$ were inferred). The others, which we term left and right outer product structure, are new. We note that for matrices with left outer-product structure, it is easier to study the exponential $S$, and use Theorem \ref{thm:main} to infer well-posedness of the corresponding BSDE, while for matrices with right outer-product structure, it is easier to study the BSDE directly, and infer properties of $S$ through Theorem \ref{thm:main}. This fact illustrates the utility of our results on the relationship between $S$ and BSDE($\kA$). While it is known that sliceability (a form of local smallness studied in \cite{Delbaen-Tang} and \cite{harter2019}) of $\kA$ guarantees that $S$ satisfies the reverse H\"older inequality, our results show that the sliceability may be omitted if the matrix satisfies a nice structural condition. Furthermore, a result from \cite{jackson2021existence} allows us to extend our results to matrices of the form $\kA + \Delta \kA$, where $\kA$ has one of the special strucures listed above and $\Delta \kA$ is sliceable. See Proposition \ref{pro:sliceable} for a precise statement.

\subsubsection{Application to quadratic systems}

The third contribution of the paper is to apply our analysis of linear BSDEs with $\bmo$ coefficients to prove existence and uniqueness results for two classes of quadratic BSDE systems. The main results of this section are Theorems \ref{thm:qlmain} and \ref{thm:udmain}, which concern drivers with two different special structures. Admittedly, the structural conditions appearing in Theorems \ref{thm:qlmain} and \ref{thm:udmain} are very specific, but we emphasize that, to the best of the author's knowledge, these are the first global well-posedness results for non-Markovian quadratic BSDEs outside of the ``diagonal" (see \cite{HuTan16}) or the ``triangular" (see \cite{jackson2021existence}) settings. By global we mean without any smallness assumption on $T$, $\xi$, of $f$. 

Theorem \ref{thm:qlmain} concerns quadratic systems whose driver $f$ has a quadratic term of the form $\kz b^T \kz$ for some $b \in \R^n$. Viewing $\kz$ as an element of $(\R^d)^n$ as explained in the notations section, this amounts to the requirement that 
\begin{align*}
    f^i(t,\omega, y, \kz) = g^i(t,\omega, y, \kz) + \kz^i \cdot (b^T \kz) 
    = g^i(t,\omega, y, \kz) + \sum_{j = 1}^n  b_j  \kz^i \cdot\kz^j,
\end{align*}
where $g$ is Lipschitz in $y$ and $\kz$. We note that the quadratic term treated here has already been treated in a Markovian setting, namely in \cite{chenam} and \cite{xing2018}. Indeed, the term $\kz b^T \kz$ is a special case of the quadratic term treated in \cite{chenam}, which takes the form $\kz h(t,\omega, y, \kz)$, where $h$ is Lipschitz in $y$ and $\kz$, and the quadratic-linear term appearing in Definition 2.10 of \cite{xing2018} is more general still. But even in the Markovian case, our results are not covered by \cite{chenam} and \cite{xing2018}, because need to assume any smoothness on the terminal condition $\xi$.

Theorem \ref{thm:udmain} is an existence and uniqueness result for BSDEs with non-linearities of the form $a h(\kz)$ for some $a \in \R^n$ and $h : (\R^d)^n \to \R$ of quadratic growth. This amounts to the requirement that the driver $f$ decomposes as 
\begin{align*}
    f^i(t,\omega, y, \kz) = g^i(t,\omega, y,\kz) + a_i h(t,\omega, y,\kz), 
\end{align*}
where $g$ is a Lipschitz driver and we write $a = (a_1,...,a_n)$. We call such drivers unidirectional, because all of the quadratic growth ``points" in the direction of the vector $a$. Both Theorem \ref{thm:qlmain} and Theorem \ref{thm:udmain} are proved by combining our new results on linear systems with the Malliavin calculus approach introduced in \cite{Briand-Elie} and applied to systems in \cite{harter2019} and \cite{jackson2021existence}. 

We note that Theorems \ref{thm:qlmain} and \ref{thm:udmain} can be generalized in various ways, and in particular the smoothness assumptions on the driver $f$ can be removed through an approximation argument. We do not pursue these generalizations in order to avoid additional technicalities and to focus on the main idea of the paper, namely the connection between the reverse H\"older inequality, linear BSDEs with $\bmo$ coefficents, and quadratic BSDE systems.

\subsection{Structure of the paper}

In the remainder of the introduction, we fix notations and conventions, and recall some relevant facts about $\BMO$ martingales and their exponentials. In section \ref{sec:counter}, we present some counterexamples which demonstrate that \eqref{lin-bsde-1} is ill-posed in general. Sections \ref{sec:existunique} and \ref{sec:rh} contain our analysis of the relationship between the well-posedness of \eqref{lin-bsde-1} and the properties of $S$. Finally, in Section \ref{sec:structural} we consider specific structural condition on $\kA$ which guarantee well-posedness.

\subsection{Notations and Preliminaries}

\subsubsection{The probabilistic setup}
Throughout the paper, $\kB$ denotes a $d$-dimensional Brownian motion defined on a probability space $(\Omega, \sF, \bP)$. We denote by $\bF = (\sF_t)$ the augmentation of the natural filtration of $\kB$, and use $\bE_{\tau}[\cdot]$ to denote $\bE[\cdot | \sF_{\tau}]$ whenever $\tau$ is a stopping time. 

\subsubsection{Universal constants}
We use $n \in \N$ to denote the dimension of the unknown process $Y$ appearing in equations \eqref{lin-bsde-1}, and $T \in (0,\infty)$ to denote the time horizon. We use $d$ to denote the dimension of the Brownian motion $B$. Throughout the paper, $n$, $d$ and $T$ are considered fixed, and constants which depend only on these parameters are called universal. We will use $A \leqc B$ to mean $A \leq CB$. When a constant $C$ depends on some parameter $a$ in addition to $n$, $d$, and $T$, we state this by writing $C = C(a)$. For example, the statement 
\begin{align*}
    A \leqc B, \, \, \const{a}
\end{align*}
means that there exists a constant $C$, depending only on $a, n, d$, and $T$, such that $A \leq CB$.

\subsubsection{Spaces of random variables and processes.} We denote by $\sP$ the space of progressively measurable processes taking values in some Euclidean space, which will be made explicit when necessary, e.g. $\sP(\R^n)$ denotes the space of progressively measurable processes taking values in $\R^n$. We denote by $\mathcal{M}$ the space of continuous local martingales taking value in some fixed Euclidean space. Now let $1 \leq p \leq \infty$, and $1 \leq q < \infty$. We will work with the following spaces 
\begin{itemize}
    \item $\lpee$ is the set of all $p$-integrable random variables, vectors, or matrices.
    \item $\spee$ is the space of all continuous, adapted processes $Y$ such that 
    \begin{align*}
        \norm{Y}_{\spee} \coloneqq \norm{\sup_{0 \leq t \leq T} |Y_t|}_{\lpee} < \infty. 
    \end{align*}
    \item $\mpee$ is the set of all uniformly integral martingales with $M_T \in \lpee$.
    \item $\BMO$ denotes the set of all $M \in \mpee$ such that $M_0 = 0$ and 
    \begin{align*}
        \norm{M}_{\BMO} \coloneqq \sup_{\tau} \norm{ \bE_{\tau}[|M_T - M_{\tau}|^2] }_{\linf}^{1/2} < \infty,
    \end{align*}
    where the supremum is taken over all stoping times $\tau$ with $0 \leq \tau \leq T$. 
        \item $\bmo$ is the set of all $Z \in \sP$ such that 
    \begin{align*}
        \norm{Z}_{\bmo} = \sup_{0 \leq \tau \leq T} \norm{
        \bE[\int_{\tau}^T |Z_t|^2 dt]  }_{\linf}^{1/2} < \infty
    \end{align*}
    where the supremum is taken over all stopping times $\tau$ with $0 \leq \tau \leq T$. 
     \item $\bmoh$ is the set of all $\beta \in \sP$ such that 
    \begin{align*}
        \norm{\beta}_{\bmoh} \coloneqq \sup_{0 \leq \tau \leq T} \norm{
        \bE[\int_{\tau}^T |\beta_t| dt]  }_{\linf} < \infty
    \end{align*}
    where the supremum is taken over all stopping times $\tau$ with $0 \leq \tau \leq T$.
    \item $\ltp$ is the set of all $Z \in \sP$ such that 
    \begin{align*}
        \norm{Z}_{\ltp} \coloneqq \bE[  \big(\int_0^T |Z_t|^2 dt\big)^{p/2} ]^{1/p} < \infty. 
    \end{align*}
    \item $\lii$ is the set of all $Z \in \sP$ such that 
    \begin{align*}
        \norm{Z}_{\lii} \coloneqq \esssup_{[0,T] \times \Omega}{|Z|} < \infty. 
    \end{align*}
\end{itemize}
In all cases, the co-domain of the process and any additional measurability restrictions will be made clear when necessary. For example, $\lpee(\sF_T ; \R^n)$ denotes the space of $p$-integrable $\R^n$-valued random variables $\xi$ which are measurable with respect to $\sF_T$. In fact, unless indicated otherwise, $\xi$ will always be used to denote an $\R^n$-valued, $\sF_T$ measurable random vector.

\subsubsection{Conventions for multi-dimensional processes.}
We follow the conventions introduced in \cite{jackson2021existence} for multi-dimensional processes. Namely, we allow our processes to take values in spaces of vectors or matrices whose entries in turn lie in the Euclidean space $\R^d$, and interpret products in a natural way. For example, in the BSDE \eqref{lin-bsde-2}, we view the coefficient matrix $\kA$ as taking values in $(\R^{d})^{n \times n}$, and $\kZ$ as taking values in $(\R^d)^{n}$. So, for $1 \leq i,j \leq n$, the $(i,j)$-entry of $\kA$, written $\kA^i_j$ takes values in $\R^d$, and the $i^{th}$ entry of $\kZ$, written $\kZ^i$ takes values in $\R^d$. The product $\kA \kZ$ takes values in $\R^n$, and is given component-wise as 
\begin{align*}
    (\kA \kZ)_i = \sum_{j = 1}^n \kA^i_j \cdot \kZ^j, 
\end{align*}
with $\kA^i_j \cdot \kZ^j$ denoting the usual inner product of the $\R^d$-valued processes $\kA^i_j$ and $\kZ^j$. 

\subsubsection{Preliminaries on BMO and the reverse H\"older inequality.}

\begin{definition} \label{sdef}
Let $M \in \m(\R^{n \times n})$. We define the \textbf{stochastic exponential of $M$}, written $\mathcal{E}(M)$, to be the unique solution $S$ to the matrix SDE 
\begin{align} \label{expdef}
    dS = S dM, \, \, S_0 = I_{n \times n},
\end{align}
where $I_{n \times n}$ denotes the $n \times n$ identity matrix. 
\end{definition}
That \eqref{expdef} has a unique solution follows from \cite{Pro04} Theorem 3.7. Notice that $\int \kA d \kB$ is an $\R^{n \times n}$-valued local martingale. In the case that $M = \int \kA d\kB$, $S$ solves the matrix SDE\begin{align*}
    dS = S \kA d\kB, \, \, S_0 = I_{n \times n}.
\end{align*}
When there is no ambiguity, we will write $S$ for the stochastic exponential of $M$ (or $\int \kA d\kB)$ without making explicit the dependence on $M$ (or $\kA$). When $n = 1$, $S$ is the usual stochastic exponential of $M$, i.e. $S = \exp{M - \frac{1}{ 2}\langle M \rangle}$. 

\begin{definition} \label{def:rp}
Given $1 \leq p < \infty$, we say that $S$ satisfies the \textbf{reverse H\"older inequality} $\rp$ if the estimate
\begin{align}  \label{rhdef}
\norm{\bE_{\tau}[|S_{\tau}^{-1}S_T|^p]}_{\linf} \leq C
\end{align}
holds for all stopping times $\tau$ with $0 \leq \tau \leq T$ and some constant $C$. Here $| \cdot |$ denotes the operator norm for matrices. Moreover, we denote by $\rpc(S)$ the smallest constant $C$ such that \eqref{rhdef} holds for each $\tau$. 
\end{definition}

The conditional H\"older's inequality shows that if $1 \leq p < p' < \infty$ and $S$ satisfies $(R_{p'})$, then $S$ also satisfies $\rp$. Notice that if $S$ is a true martingale that satisfies $\rp$ then $S$ is necessarily in $\mathcal{M}^p$. In fact, the following stronger estimate holds. 

\begin{proposition} \label{pro:strongerrh}
Suppose that $S$ is a true martingale and satisfies $\rp$ for some $p > 1$. Then, $S$ satisfies the stronger estimate 
\begin{align} \label{strongrh}
    \norm{\bE_{\tau}[ \sup_{\tau \leq t \leq T} |S_{\tau}^{-1} S_t|^p]}_{\linf} \leq \big(\frac{p}{p-1}\big)^p \rpc(S)
\end{align}
for all stopping times $\tau$ with $ 0 \leq \tau \leq T$. 
\end{proposition}

\begin{proof}
Fix a stopping time $\tau$ and an event $F \in \sF_{\tau}$. The process $ M \coloneqq 1_{F} S_{\tau}^{-1} S$ is a martingale on the stochastic interval $[\tau, T]$. Moreover, 
\begin{align*}
    \bE[ |M_T|^p] = \bE[ |1_F S_{\tau}^{-1}S_T|^p] = \bE[ 1_F \bE_{\tau}[|S_{\tau}^{-1}S_T|^p] \leq \rpc(S)\bP[F]
\end{align*}
It follows from Doob's maximal inequality that 
\begin{align*}
    \bE[ 1_F \sup_{\tau \leq u \leq T} |S_{\tau}S_u|^p] = \bE[ \sup_{\tau \leq u \leq T} |M_u|^p] \leq \big(\frac{p}{p-1}\big)^p \bE[|M_T|^p] \leq \big(\frac{p}{p-1}\big)^p \rpc(S) \bP[F]. 
\end{align*}
The desired estimate follows immediately.
\end{proof}

The following Proposition is adapted from Theorem 3.1 of \cite{Delbaen-Tang}, and the proof is the same as Corollary 3.2 of \cite{Kazamaki}. 

\begin{proposition} \label{pro:largerp}
Suppose that $S$ is a true martingale which satisfies $\rp$, for some $p > 1$. Then there exists $p' > p$ such that $S$ satisfies $(R_{p'})$. 
\end{proposition}

The following well-known result, which combines Theorems 2.3 and 3.1 of \cite{Kazamaki}), is one of the reasons that $\BMO$ martingales are so usefel. 

\begin{theorem} \label{thm:1drh}
 If $n = 1$ and $M \in \BMO$, then $S$ is a uniformly integrable martingale which satisfies $\rp$ for some $p = p(\norm{\kA}_{\bmo}) > 1$. 
\end{theorem}

\subsubsection{Definitions of well-posedness.}
In addition to \eqref{lin-bsde-1},
We will also consider the homogenous version of \eqref{lin-bsde-1}, namely the BSDE 
\begin{align} \label{lin-bsde-2}
    Y = \xi + \int_{\cdot}^{T} \kA \kZ dt - \int_{\cdot}^{T} \kB. 
\end{align}
Given $\kA \in \bmo((\R^d)^{n \times n})$, we refer to the equation \eqref{lin-bsde-1} as BSDE($\kA$) and \eqref{lin-bsde-2} as HBSDE($\kA$), the H indicating the homogenous version of the equation. 
The following definitions will give a precise meaning to the well-posedness of \eqref{lin-bsde-1} and \eqref{lin-bsde-2}. 

\begin{definition} \label{def:wellposed}
Let $\kA \in \bmo((\R^{d})^{n \times n})$. We say that BSDE($\kA$) is \textbf{well-posed in $\sinf$} if for each $(\xi, \beta) \in \linf \times \loi$ there is a unique solution $(Y,\kZ) \in \sinf \times \bmo$ to \eqref{lin-bsde-1}, which satisfies the estimate
\begin{align*}
    \norm{Y}_{\sinf} + \norm{\kZ}_{\bmo} \leqc \norm{\xi}_{\sinf} + \norm{\beta}_{\loi}, \, \, \const{\kA}. 
\end{align*}
We say that HBSDE($\kA$) is \textbf{well-posed in $\sinf$} if for each $\xi \in \linf$, there is a unique solution $(Y,\kZ) \in \sinf \times \bmo$ to \eqref{lin-bsde-2}, which satisfies the estimate
\begin{align*}
    \norm{Y}_{\sinf} + \norm{\kZ}_{\bmo} \leqc \norm{\xi}_{\sinf}, \, \, \const{\kA}. 
\end{align*}
For $1 < q <\infty$, we say that BSDE($\kA$) is \textbf{well-posed in $\sque$} if for each $\xi \in \lque$ and $\beta \in \loq$, there is a unique solution $(Y,\kZ) \in \sque \times \ltq$ to \eqref{lin-bsde-1}, which satisfies the estimate 
\begin{align*}
    \norm{Y}_{\sque} + \norm{\kZ}_{\ltq} \leq_{C_q} \norm{\xi}_{\lque } + \norm{\beta}_{\loq}, \, \, C_q = C_q(\kA). 
\end{align*}
We say that HBSDE($\kA$) is \textbf{well-posed in $\sque$} if for each $\xi \in \lque$, there is a unique solution $(Y,\kZ) \in \sque \times \ltq$ to \eqref{lin-bsde-2} which satisfies the estimate 
\begin{align*}
    \norm{Y}_{\sinf} + \norm{\kZ}_{\ltq} \leq_{C_q} \norm{\xi}_{\lque }, \, \, C_q = C_q(\kA). 
\end{align*}
\end{definition}
If BSDE($\kA$) is well-posed in $\sque$, then there is a bounded solution operator $S_{\kA} : \lque \times \loq \to \sque \times \ltq$ (or $\linf \times \loi \to \sinf \times \bmo$ if $q = \infty$). We denote the by $\opnorm{S_{\kA}}_q$ the norm of this operator. That is, when $q < \infty$, $\opnorm{S_{\kA}}_q$ is the smallest constant $C$ verifying the estimate 
\begin{align*}
    \norm{Y}_{\sque} + \norm{\kZ}_{\ltq} \leqc \norm{\xi}_{\lque} + \norm{\beta}_{\ltq}
\end{align*}
for each $(\xi, \beta) \in \lque \times \loq$, where $(Y,\kZ) \in \sque \times \bmo$ is the unique solution to \eqref{lin-bsde-1}. A similar description holds for $\opnorm{S_{\kA}}_{\sinf}$. 
Of course, if BSDE($\kA$) is well-posed in $\sque$, then HBSDE($\kA$) is well-posed in $\sque$ as well.

There is one more notion of well-posedness, which, for technical reasons, will be convenient to define. 
\begin{definition}
We say that BSDE($\kA$) is \textbf{strongly well-posed in $\sinf$} if for each $\xi \in \linf$ and $\beta \in \bmoh$, there is a unique solution $(Y,\kZ) \in \sinf \times \bmo$ to \eqref{lin-bsde-1}, which satisfies the estimate 
\begin{align*}
    \norm{Y}_{\sinf} + \norm{\kZ}_{\ltq} \leqc \norm{\xi}_{\linf} + \norm{\beta}_{\bmoh}, \, \, \const{\kA}. 
\end{align*}
\end{definition}
If BSDE($\kA$) is strongly well-posed in $\sinf$, we again denote by $S_{\kA}$ the bounded solution map $\linf \times \bmoh \to \sinf \times \bmo$, and denote its norm by $\opnorm{S_{\kA}}_{\infty, s}$.
Since $\loi \subset \bmoh$, and $\norm{\beta}_{\bmoh} \leq \norm{\beta}_{\loi}$, it is clear that if BSDE($\kA$) is strongly well-posed in $\sinf$ then it is well-posed in $\sinf$, with $\opnorm{S_{\kA}}_{\infty} \leq \opnorm{S_{\kA}}_{\infty,s}$.

Theorem \ref{thm:1drh}, together with a well-known change of measure technique, can be used to prove the following.

\begin{theorem} \label{thm:1dwellposed}
Suppose that $n = 1$ and $\kA \in \bmo$. Then there exists $q^* = q^*(\norm{\kA}_{\bmo})$ such that BSDE($A$) is well-posed in $\sque$ for each $q > q^*$. Moreover, the constants appearing in the definition of well-posedness depend only on $\norm{\kA}_{\bmo}$. 
\end{theorem}

\section{Counterexamples} \label{sec:counter}

In this section, we show that no analogue of Theorem \ref{thm:1drh} or Theorem \ref{thm:1dwellposed} can hold when $n > 1$. That $M \in \BMO(\R^{n \times n})$ does not imply that $S$ is a true martingale was demonstrated first by \'Emery in \cite{Emery}. Because we will be using the example in \cite{Emery} as a starting point to build a nonexistence example, we present it in detail. 

\begin{example}[\cite{Emery}]
Recall that complex numbers can be identified with $2 \times 2$ real matrices of the form 
$\begin{pmatrix}
x & y \\
-y & x
\end{pmatrix}$
via the isomorphism
\begin{align*}
    z = x + iy \mapsto \Phi(z) \coloneqq \begin{pmatrix}
    x & y \\ 
    -y & x
    \end{pmatrix}.
\end{align*}
\'Emery's first step was the to observe that if $N$ is a complex-valued martingale, then $M = \Phi(N)$ is a $\R^{2 \times 2}$ valued martingale whose stochastic exponential $S$ satisfies
\begin{align} \label{commute}
    S  = \Phi(\mathcal{E}(N)) = \Phi(\exp{N_t - \frac{1}{2} \langle N \rangle_t}). 
\end{align}
That is, the stochastic exponential commutes with the isomorphism $\Phi$. 
This trick gives a way to compute $\mathcal{E}(M)$ when $M = \Phi(N)$ for a complex valued martingale $N$.In particular, if $B$ is a one-dimensional Brownian motion, $\tau = \inf \{t \geq 0 : |B_t| = \frac{\pi}{2}\}$, and $N = i B^{\tau}$, then we can compute explicitly 
\begin{align} \label{ncomp}
    \mathcal{E}(N) = \exp{iB^{\tau} + \frac{1}{2} \langle B^{\tau} \rangle}.
\end{align}
Setting $M = \Phi(N) = \begin{pmatrix}
0 & B^{\tau} \\
-B^{\tau} & 0, 
\end{pmatrix}$
the commutativity relation \eqref{commute} together with \eqref{ncomp} reveals 
\begin{align} \label{sformula}
    S_t = \exp{\frac{\tau \wedge t}{2}} \begin{pmatrix}
    \cos(B_{\tau \wedge t}) & \sin(B_{\tau \wedge t}) \\
    - \sin(B_{\tau \wedge t}) & \cos(B_{\tau \wedge t}), 
    \end{pmatrix} \nonumber \\
    S_{\infty} =    \exp{\frac{\tau}{2}} \begin{pmatrix} 
    0 & \sin(B_{\tau}) \\
    - \sin(B_{\tau}) & 0. 
    \end{pmatrix}
\end{align}
In particular, since $S_0$ is the identity and $S_{\infty}$ is zero on the diagonal, we can not have $S_0 = \bE[S_{\infty}]$, so $S$ is not a uniformly integrable martingale. A time change argument lets us produce an example in which $S$ is a strict local martingale.

In fact, more can be said. It follows from the proof of Theorem 1.7 in \cite{Kazamaki} that $\bE[|S_{\infty}|] = \bE[\exp{\frac{\tau}{2}}] = \infty$. In particular, $S$ does not satisfy $(R_1)$. 

\end{example} 

\'Emery's example demonstrates that if $M \in \BMO(\R^{n \times n})$, then in general $S$ may be a strict local martingale, and need not satisfy any $\rp$. Because the reverse H\"older inequality in dimension one (Theorem \ref{thm:1drh}) is instrumental in proving the well-posedness of \eqref{lin-bsde-2} in dimension one (Theorem \ref{thm:1dwellposed}), \'Emery's example suggests that equation \eqref{lin-bsde-2} may be ill-posed when $n > 1$. Indeed, in example 2.3 of \cite{jackson2021existence} a matrix $A \in \bmo(\R^{2 \times 2})$ is constructed such that the equation 
\begin{align*}
    Y = \int_{\cdot}^T A Z dt - \int_{\cdot}^T Z dB 
\end{align*}
has a nonzero solution. Thus uniqueness for $\eqref{lin-bsde-2}$ can fail when $n > 1$.

In the remainder of this section, we produce an $A \in \bmo(\R^{2 \times 2})$ and a $\xi \in \linf$ such that \eqref{lin-bsde-2} has no solution $(Y,\kZ) \in \sinf \times \bmo$. That is, we show that existence may also fail.

\begin{example}
Let $d = 1$, and define a local martingale $N$ on $[0,T)$ by 
\begin{align*}
    N_t = \int_0^t f(s) dB_s, \\
    f(s) = \begin{cases}
    0 & 0 \leq s \leq T/2 \\
    \frac{1}{T - s} & T/2 \leq s < T
    \end{cases}
\end{align*}
So, $N = 0$ on $[0,T/2]$, and on $[T/2, T)$, $N$ is the image of a Brownian motion on $[0,\infty)$ under a deterministic time change. Next, choose a sequence of numbers $\{b_k\}_{k \in \N}$ with the following properties
\begin{enumerate}
    \item $0 \leq b_k < \frac{\pi}{2}$, $b_k \uparrow \frac{\pi}{2}$, 
    \item $\sum_{k = 1}^{\infty} \frac{1}{2^k \cos(b_k)} = \infty$,
    \item $\frac{1}{2^k \cos(b_k)} \to 0$ as $k\to \infty$. 
\end{enumerate}
Let $\{A_k\}$ be a partition of $\sF_{T/2}$ such that $\bP[A_k] = \frac{1}{2^k}$. Set $\tau_k = \inf \{t \geq T/2 : |N_t| = b_k\} \wedge T$, and $\tau = \sum_k 1_{A_k} \tau_k$. Then $\tau$ is a stopping time, and we can define
\begin{align*}
    M = \Phi(iN^{\tau}) = \begin{pmatrix}
    0 & N^{\tau} \\
    -N^{\tau} & 0
    \end{pmatrix}.
\end{align*}
As in Kazamaki's example, we can compute the stochastic exponential $S$ of $M$ explicitly as 
\begin{align*}
    S_t = \exp{\frac{1}{2}\int_0^{\tau \wedge t} |f(s)|^2 ds} \begin{pmatrix} \cos(N_{\tau \wedge t}) & \sin(N_{\tau \wedge t}) & \\
    - \sin(N_{\tau \wedge t}) & \cos(N_{\tau \wedge t})
    \end{pmatrix}
\end{align*}
Now, let $A$ be such that $M = \int A dB$. Since $M$ is bounded, $A \in \bmo(\R^{2 \times 2})$. Set $\xi = (\cos(N_{\tau}), \sin(N_{\tau}))^{T} \in \linf(\sF_T; \R^2)$. We will now show that the equation 
\begin{align*}
    Y = \xi - \int_{\cdot}^T A Z dt - \int_{\cdot}^T Z dB
\end{align*}
has no solution $(Y,Z) \in \sinf \times \bmo$. Indeed, suppose $(Y,Z)$ is such a solution. A time change argument, together with Lemma 1.3 of \cite{Kazamaki} shows that 
\begin{align}  \label{tauj}
    \bE[\exp{\frac{1}{2} \int_0^{\tau_j} f^2(s)ds}] = \frac{1}{\cos{b_j}}
\end{align}
Set $\sigma_j = \tau \wedge \tau_j$. The first component of $S_{\sigma_j} \xi$ is given by 
\begin{align}
    \label{firstcomp} \exp{\frac{1}{2} \int_0^{\sigma_j}|f(s)|^2 ds} \big(\cos(N_{\tau \wedge \tau_j}), \sin(N_{\tau \wedge \tau_j})\big) \big(\cos(N_{\tau}), \sin(N_{\tau}) \big)^T, 
\end{align}
which is greater than $\frac{1}{2} \exp{\frac{1}{2} \int_0^{\sigma_j} |f(s)|^2 ds }$ as soon as $j$ is sufficiently large. Notice that 
\begin{align*}
    \bE[\exp{\frac{1}{2} \int_0^{\sigma_j} |f(s)|^2 ds}] \geq \bE[\sum_{k = 1}^j 1_{A_j} \exp{\frac{1}{2}\int_0^{\tau_k} |f(s)|^2 ds}] \\ = \sum_{k = 1}^j \bP[1_{A_j}] \bE[\exp{\frac{1}{2}\int_0^{\tau_k} |f(s)|^2 ds}] 
    = \sum_{k = 1}^j \frac{1}{2^k \cos(b_k)},
\end{align*}
where the last equality comes from \eqref{tauj}. Since $\sum_{k = 1}^{\infty} \frac{1}{2^k \cos(b_k)} = \infty$, we conclude from \eqref{firstcomp} that $|\bE[S_{\sigma_j}\xi]| \to \infty$ as $j \to \infty$.

An application of It\^o's formula shows that $SY$ is local martingale, and it is easy to show that for each $j$, $S^{\sigma_j}$ is a true martingale, and hence so is $S^{\sigma_j}Y^{\sigma_j}$. In particular, we have

\begin{align} \label{y0comp}
    Y_0 = \bE[S_{\sigma_j} Y_{\sigma_j}]. 
\end{align}
But we also have 
\begin{align*}
    |Y_0 - \bE[S_{\sigma_j} \xi]| = |\bE[S_{\sigma_j}(Y_{\sigma_j} - \xi)]|
    \leqc \bE[|S_{\tau_j}|1_{\tau_j < \tau}]  
    \\
    = \bE[\exp{\frac{1}{2} \int_0^{\tau_j \wedge \tau} |f(s)|^2 ds} 1_{\cup_{k > j} A_k}] \\
    \leq \bE[\exp{\frac{1}{2} \int_0^{\tau_j} |f(s)|^2 ds} 1_{\cup_{k > j} A_k}] \\
    = \bE[\exp{\frac{1}{2} \int_0^{\tau_j} |f(s)|^2 ds}]
    \sum_{k = j + 1}^{\infty} \frac{1}{2^k} 
\leq  \frac{1}{2^j\cos(b_j)} ,
\end{align*}
which tends to zero as $j \to \infty$. Since $Y_0$ is fixed and finite, while $|\bE[S_{\sigma_j}\xi]| \to \infty$ as $j \to \infty$, we have reached a contradiction, and conclude that no solution $(Y,Z) \in \sinf \times \bmo$ can exist. 

\end{example}

\section{Sufficient Conditions for Existence and Uniqueness} \label{sec:existunique}

In this section, we show how existence and uniqueness for \eqref{lin-bsde-1} can be inferred from properties of $S$.

\subsection{Sufficient Conditions for Uniqueness}

In this sub-section, we give sufficient conditions for uniqueness to hold for \eqref{lin-bsde-1}, in terms of the exponential $S$. To be precise, we make the following definitions.
\begin{definition} 
We say that \textbf{uniqueness in $\sinf$} holds for BSDE($\kA$) if for each $(\xi, \beta) \in \linf \times \loi$, there is at most one solution $(Y,\kZ) \in \sinf \times \bmo$ to \eqref{lin-bsde-1}. For $1 < q < \infty$, we say \textbf{uniqueness in $\sque$} holds for BSDE($\kA$) if for each $(\xi, \beta) \in \lque \times \loq$, there is at most one solution $(Y,\kZ) \in \sque \times \ltq$ to \eqref{lin-bsde-1}. We define uniqueness in $\sque$ for HBSDE($\kA$) analogously. 
\end{definition}

As motivation for the arguments below, first consider the case $n = 1$. Then any solution $(Y,\kZ)$ to \eqref{lin-bsde-1} satisfies 
\begin{align*}
    Y_t = \xi + \int_t^T \beta_s ds - \int_t^T \kZ_s d\kB^{\kA}, 
\end{align*}
where
\begin{align*}
    \kB^{\kA} = \kB - \int \kA dt.
\end{align*}
Furthermore, by Theorem \ref{thm:1drh}, $S = \mathcal{E}(\int \kA d\kB)$ is a true martingale, and by Girsanov's Theorem $Y$ is a local martingale martingale under the probability measure $\bP^{\kA}$, defined by 
\begin{align*}
    \frac{d\bP^{\kA}}{d\bP} = S_T.
\end{align*}
If we can verify that $Y$ is a true martingale under $\bP^{\kA}$, it would follow that
\begin{align*}
    Y_t = \bE_t^{\kA}[\xi + \int_t^T \beta_s ds] = S_t^{-1}\bE_t[S_T \big(\xi + \int_t^T \beta_s ds\big)],
\end{align*}
where $\bE_t^{\kA}[\cdot]$ denotes conditional expectation with respect to the measure $\bP^{\kA}$. This representation formula immediately implies uniqueness. In higher dimensions, the change-of-measure technique is no longer available (unless the matrix $\kA$ is diagonal), but the representation formula still holds, as long as we have enough regularity on $S$. This is the strategy we take to prove uniqueness. 

We begin with two lemmas. The first is Lemma 2.20 of \cite{jackson2021existence}, and the second is an easy application of It\^o's Lemma. 

\begin{lemma} \label{lem:sinv}
If $\kA \in \bmo$, then for each $t \in [0,T]$, $S_t$ is invertible a.s. 
\end{lemma}

\begin{lemma} \label{lem: sylocalmart}
If $(Y,\kZ)$ solves \eqref{lin-bsde-2}, then the process $S\big(Y + \int_0^{\cdot} \beta dt\big)$ is a local martingale. 
\end{lemma}

With these lemmas in hand, we can give sufficient conditions for uniqueness. 

\begin{proposition} \label{pro: yrep1}
  Suppose that $S$ is a true martingale, $(\xi, \beta) \in \linf \times \loi$, and $(Y,\kZ) \in \sinf \times \bmo$ solves \eqref{lin-bsde-1}. Then, for each $t$,we have 
  \begin{align*}
      Y_t = S_t^{-1} \bE_t[S_T\big(\xi + \int_0^{\cdot} \beta dt\big)]. 
  \end{align*}
  In particular, uniqueness in $\sinf$ holds for \eqref{lin-bsde-1}. 
\end{proposition}

\begin{proof}
From Lemma \ref{lem: sylocalmart}, $S\big(Y + \int_0^{\cdot} \beta dt\big)$ is a local martingale. Since $S$ is a true martingale, $Y \in \sinf$, and $\int_0^{\cdot} \beta dt \in \sinf$, $S\big(Y + \int_0^{\cdot} \beta dt\big)$ is of class (DL), hence a true martingale. Thus for each $t$, we have 
\begin{align*}
S_t\big(Y_t + \int_0^t \beta_u du\big) = \bE_t[ S_T\big(\xi + \int_0^T \beta_u du\big)]. 
\end{align*}
Since $S_t$ is invertible by Lemma \ref{lem:sinv}, we get 
\begin{align*}
    Y_t + \int_0^t \beta_u du = S_t^{-1} \bE_t[S_T(\xi + \int_t^T \beta_u du\big)] + S_t^{-1} \bE_t[S_T\int_0^t \beta_u du] \\
    = \bE_t[S_T(\xi + \int_t^T \beta_u du\big)] + S_t^{-1} \bE_t[S_T] \int_0^t \beta_u du \\ = \bE_t[S_T(\xi + \int_t^T \beta_u du\big)] + \int_0^t \beta_u du, 
\end{align*}
where the last line uses the adaptedness of $\beta$ and the fact that $S$ is a martingale. The result follows. 

\end{proof}

\begin{proposition} \label{pro:yrep}
  Suppose that $S \in \mpee$ for some $p > 1$, and $(\xi, \beta) \in \sque \times \loq$, where $q$ be the conjugate of $p$. Then if  $(Y,\kZ) \in \sque \times \ltq$ solves \eqref{lin-bsde-2}, then for each $t$, we have
  \begin{align*}
      Y_t = S_t^{-1}\bE_t[S_T\big(\xi + \int_0^t \beta_u du\big)]. 
  \end{align*}
  In particular, uniqueness in $\sque$ holds for BSDE($\kA$). 
\end{proposition}

\begin{proof}

If $Y \in \sque$, and $S \in \mpee$, then 
\begin{align*}
    \bE[\sup_{0 \leq t \leq T} |S_t\big(Y_t + \int_0^t \beta_u du\big)|] \leq \norm{S}_{\spee} \big(\norm{Y}_{\sque} + \norm{\beta}_{\loq}\big) < \infty.
\end{align*}
Thus $S\big(Y + \int_0^{\cdot} \beta dt\big)$ is a true martingale, and the result follows as in the proof of Proposition \ref{pro: yrep1}. 
\end{proof}

\subsection{Sufficient Conditions for Existence}

In the previous section, we saw that uniqueness for \eqref{lin-bsde-1} is related to \textit{regularity} of $S$, but does not actually require $S$ to satisfy any reverse H\"older inequality. For existence, however, we will need to verify that the candidate solution $Y_t = S_t^{-1} \bE[S_T(\xi + \int_0^{t} \beta_u du\big)]$ is sufficiently regular, and so we will need $S$ to satisfy a reverse H\"older inequality.

\begin{remark} \label{rmk:errors}
We note that the condition for existence derived here are similar in spirit to Theorem 3.2 of \cite{Delbaen-Tang}. There are two minor errors in their paper; however, which affect some of the results in Section 3. Firstly, the authors use the inequality
\begin{align*}
    \bE_{\tau}[|S_T|^p] \leqc |S_{\tau}|^p
\end{align*}
in Definition 3.1 to define the reverse H\"older inequality, but then use the inequality \eqref{rhintro} (which is stronger when $n > 1$) in the proof of Theorem 3.2 and elsewhere. There also seems to be a technical issue in the proof of Theorem 3.2, namely the proof shows how to build a solution $(\tilde{Y}, Z)$ with $\tilde{Y}$ defined by (3.7), and shows that this solution satisfies the estimate in (3.6). It is then stated (in the last line of the proof) that uniqueness follows. It is not explained, however, how to rule out the possibility of the existence of another solution $Y$ which does not satisfy the representation (3.7) or the estimate (3.6). We could not find a way to complete this argument without assuming that $S$ is a true martingale (and in this case Proposition \ref{pro:yrep} gives a simpler argument for uniqueness). 
\newline \newline
Because of these subtle issues and because we need to extend the results of \cite{Delbaen-Tang} in various ways in order to apply them to quadratic BSDEs in Section \ref{sec:qbsde}, we give a full treatment of existence via reverse H\"older, even though the main ideas come from Section 3 of \cite{Delbaen-Tang}.
\end{remark}

\begin{definition}
We say that \textbf{existence in $\sinf$} holds for BSDE($\kA$) if for each $(\xi, \beta) \in \sinf \times \loi$, there exists a solution $(Y,\kZ) \in \sinf \times \bmo$ satisfying 
\begin{align} \label{sinfexistdef}
    \norm{Y}_{\sinf} \times \norm{\kZ}_{\bmo} \leqc \norm{\xi}_{\linf} + \norm{\beta}_{\loi} \, \, \const{\kA}. 
\end{align}
We say that \textbf{strong existence in $\sinf$} holds for BSDE($\kA$) if for each $(\xi, \beta) \in \sinf \times \bmoh$, there exists a solution $(Y,\kZ) \in \sinf \times \bmo$ satisfying 
\begin{align} \label{sinfexistsdef}
    \norm{Y}_{\sinf} \times \norm{\kZ}_{\bmo} \leqc \norm{\xi}_{\linf} + \norm{\beta}_{\bmoh} \, \, \const{\kA}. 
\end{align}
We say that \textbf{existence in $\sque$} holds for BSDE($\kA$) if for each $(\xi, \beta) \in \lque \times \loq$, there exists a solution $(Y,\kZ) \in \sque \times \ltq$ satisfying 
\begin{align} \label{sqexistdef}
    \norm{Y}_{\sque} + \norm{\kZ}_{\ltq} \leq_{C_q} \norm{\xi}_{\sque} + \norm{\beta}_{\ltq} \, \, C_q = C_q(\kA).
\end{align}
We define existence in $\sque$ for HBSDE($\kA$) similarly. 
\end{definition}

\begin{proposition} \label{pro:linfexisth}
Suppose that $S$ satisfies $\ro$. Then existence in $\sinf$ holds for HBSDE($\kA$). In particular, for each $\xi \in \linf$, there is a solution $(Y,\kZ)$ satisfying 
\begin{enumerate}
    \item $Y_t = S_t^{-1}\bE[S_T \xi] \text{ for each } t \in [0,T]$
    \item $\norm{Y}_{\sinf} \leq_{\roc(S)}\norm{\xi}_{\linf}$
    \item $\norm{\kZ}_{\bmo} \leqc \norm{\xi}_{\linf}, \, \, \const{\kA, \roc(S)}$. 
\end{enumerate}

\end{proposition}

\begin{proof}
Set $Y_t  = S_t^{-1}\bE_t[S_T\xi]$ and notice that
\begin{align*}
    |Y_t| \leq \roc(S)\norm{\xi}_{\linf}
\end{align*}
It follows that $\norm{Y}_{\sinf} \leq \roc(S) \norm{\xi}_{\linf}$. The process $S^{-1}$ has the dynamics
\begin{align} \label{sinv}
    dS^{-1} = \kA^2 S^{-1} dt - \kA S^{-1} d\kB,
\end{align}
or more precisely the matrix SDE 
\begin{align*}
    dX = \kA^2 X dt - \kA X d\kB, \, \, X_0 = I_{n \times n}
\end{align*}
has a unique solution which satisfies $X_t = S_t^{-1}$ a.s. for each $t$, i.e. we may choose a version of $S^{-1}$ satisfying \eqref{sinv}. Now let $\tilde{\kZ} \in \lto$ be such that 
\begin{align*}
    \bE_t[S_T \xi] = \int_0^t \tilde{\kZ} d\kB.
\end{align*}
Applying It\^o's formula and using \eqref{sinv}, shows that 
\begin{align*}
    dY = - \kA \kZ dt + \kZ d\kB,
\end{align*}
where $\kZ = - \kA S^{-1} \bE_{\cdot}[S_T\xi] + S^{-1} \tilde{\kZ} \in \lto_{loc}$. So far, we have shown that there exists $(Y,\kZ) \in \sinf \times \lto_{loc}$ satisfying \eqref{lin-bsde-2} and such that $\norm{Y}_{\sinf} \leq_{\roc(S)} \norm{\xi}_{\linf}$.
It remains to verify that $\kZ \in \bmo$, and estimate its norm. First, suppose that $\kZ \in \lto$, and notice that 
\begin{align*}
    |Y_t|^2 = |\xi|^2 + 2 \int_t^T Y_s \kA_s \kZ_s ds - \int_t^T |\kZ_s|^2 ds - 2 \int_t^T Y_s \kZ_s d\kB_s. 
\end{align*}
Since $\int Y \kZ d \kB$ is a true martingale, we conclude that 
\begin{align*}
    \bE_t[ \int_t^T |\kZ_s|^2 ds] \leq \norm{\xi}_{\linf}^2 + 2 \bE_t[\int_t^T |Y_s \kA_s \kZ_s| ds] \\
    \leq \norm{\xi}_{\linf}^2 + 2\norm{Y}_{\sinf}\bE_t[\int_t^T |\kA_s \kZ_s| ds] \\
    \leq \norm{\xi}_{\linf}^2 + 2\norm{Y}_{\sinf}\big(\bE_t[\int_t^T |\kA_s|^2 ds]\big)^{1/2}\big(\bE_t[\int_t^T |\kZ_s|^2 ds]\big)^{1/2} \\ 
    \leq  \norm{\xi}_{\linf}^2 + \norm{Y}^2_{\sinf}\bE_t[\int_t^T |\kA_s|^2 ds] + \frac{1}{2} \bE_t[\int_t^T |\kZ_s|^2 ds] \\
    \leq \norm{\xi}_{\linf}^2 + \norm{Y}_{\sinf}^2 \norm{\kA}^2_{\bmo} + \frac{1}{2} \bE_t[\int_t^T |\kZ_s|^2 ds].
\end{align*}
Thus 
\begin{align*}
    \frac{1}{2} \bE_t[ \int_t^T |\kZ_s|^2 ds] \leq \norm{\xi}_{\linf}^2 + \norm{Y}_{\sinf}^2 \norm{\kA}^2_{\bmo}
\end{align*}
That $\kZ \in \bmo$, as well as the desired estimate on $\norm{\kZ}_{\bmo}$ follows immediately from the estimate on $\norm{Y}_{\sinf}$. That in fact $\kZ \in \lto$ can be established using a localization argument, together with the a-priori estimate we have just established. 
\end{proof}

\begin{proposition} \label{pro:linfexist}
Suppose that $S$ satisfies $\ro$ and that $S$ is a true martingale. Then existence in $\sinf$ holds for BSDE($\kA$). In particular, for each $(\xi, \beta) \in \sinf \times \loi$, there is a solution $(Y,\kZ)$ satisfying
\begin{enumerate}
    \item $Y_t = S_t^{-1}\bE_t[S_T\big(\xi + \int_t^T \beta_u du] \text{ for each } t \in [0,T]$,
    \item $\norm{Y}_{\sinf} \leq_{\roc(S)} \norm{\xi}_{\linf} + \norm{\beta}_{\loi}$,
    \item  $\norm{Z}_{\bmo} \leqc \norm{\xi}_{\linf} + \norm{\beta}_{\loi}, \, \, \const{\norm{\kA}_{\bmo}, \roc(S)}$.
\end{enumerate}
\end{proposition}

\begin{proof}
Set $Y_t =  S_t^{-1}\bE_t[S_T \big(\xi + \int_t^T \beta_u du\big)]$ and notice that
\begin{align*}
    |Y_t| \leq \roc(S)\big( \norm{\xi}_{\linf} + \norm{\beta}_{\loi}\big).
\end{align*}
Now, let $\tilde{\kZ}$ be such that 
\begin{align*}
    \bE_t[S_T(\xi + \int_0^T \beta_s ds)] = \int_0^t \tilde{\kZ}_s d\kB_s
\end{align*}
It\^o's formula, together with the dynamics of $S^{-1}$ (given by \eqref{sinv}) show that 
\begin{align*}
    dY = -\big(\kA \kZ + S^{-1}\bE_{\cdot}[S_T \beta_{\cdot}]\big)dt + \kZ d\kB = - \big(\kA \kZ + \beta\big)dt + \kZ d\kB
\end{align*}, 
where $\kZ = - \kA S^{-1} \bE_{\cdot}[S_T \xi] + S^{-1} \tilde{\kZ}$, and the last equality we have used the fact that $S$ is a true martingale. It remains only to show that $\kZ \in \bmo$ and estimate $\norm{\kZ}_{\bmo}$, which can be done exactly as in the proof of Proposition \ref{pro:linfexist}. 
\end{proof}

\begin{remark}
An inspection of the above proof reveals why we need the assumption that $S$ is a true martingale in Proposition \ref{pro:linfexist} and not \ref{pro:linfexisth}. Namely, It\^o's formula shows that the candidate solution $Y = S^{-1}\bE_{\cdot}[S_T \big(\xi + \int_{\cdot}^T \beta dt]$ solves 
\begin{align*}
    dY = -\big(\kA \kZ + S^{-1}\bE_{\cdot}[S_T \beta_{\cdot}]\big)dt
\end{align*}
for an appropriate choice of $\kZ$. If we do not assume that $S$ is a true martingale, it may not be the case that  $S^{-1}\bE_{\cdot}[S_T \beta_{\cdot}] = \beta$, in which case the candidate solution $Y$ actually solves the wrong equation.
\end{remark}

Very similar arguments give the following versions of the previous two propositions.

\begin{proposition} \label{pro:lqexisth}
Suppose that $S$ satisfies $\rp$ for some $p > 1$. Let $p' \in (1,p)$ and $q$ be the conjugate of $p'$. Then existence in $\sque$ holds for HBSDE($\kA$). In particular, for each $\xi \in \lque$, there is a solution $(Y, \kZ)$ satisfying 
\begin{enumerate}
    \item $Y_t = S_t^{-1}\bE_t[S_T\xi] \text{ for each } t \in [0,T]$,
    \item $\norm{Y}_{\sque} + \norm{\kZ}_{\ltq} \leqc \norm{\xi}_{\lque}, \, \, \const{\rpc(S), \norm{\kA}_{\bmo}}$.
\end{enumerate}
\end{proposition}

\begin{proposition} \label{pro:lqexist}
Suppose that $S$ satisfies $\rp$ for some $p > 1$ and that $S$ is a true martingale. Let $p' \in (1,p)$ and $q$ be the conjugate of $p'$. Then existence in $\sque$ holds for BSDE($\kA$). In particular, for each $\xi \in \lque$, there is a solution $(Y, \kZ)$ satisfying 
\begin{enumerate}
    \item $Y_t = S_t^{-1}\bE_t[S_T\big(\xi + \int_t^T \beta_u du\big)] \text{ for each } t \in [0,T]$,
    \item $\norm{Y}_{\sque} + \norm{\kZ}_{\ltq} \leqc \norm{\xi}_{\lque}  + \norm{\beta}_{\loq}, \, \, \const{\rpc(S), \norm{\kA}_{\bmo}}$.
\end{enumerate}
\end{proposition}

\section{Reverse H\"older from Well-posedness} \label{sec:rh}

The results of the previous section show that if $S$ is a true martingale satisfying $\rp$, then BSDE($\kA$) is well-posed in $\sque$, where $1 \leq p < \infty$ and $q$ is the conjugate of $p$. This section is devoted to proving a converse of this statement. We start with a sequence of simple lemmas. Throughout this section, $\sG$ denotes a $\sigma$-algebra contained in $\sF$.

\begin{lemma} \label{vectors}
Let $1 \leq p < \infty$ and $q$ be the conjugate of $p$. Let $X \in \lpee(\Omega; \R^n)$. Then $\norm{\bE[|X|^p | \sG}_{\linf}^{1/p}$ is equal to the smallest constant $C$ such that
\begin{align} \label{assump}
\norm{\bE[ X \cdot Y| \sG]|}_{\lque} \leqc \norm{Y}_{\lque}
\end{align}
holds for all $Y \in \lque(\Omega ; \R^n)$. 
\end{lemma}

\begin{proof} When $p > 1$, we have by the conditional H\"older inequality, we have
\begin{align*}
    \bE[X \cdot Y | \sG] \leq \bE[|X|^p | \sG]^{1/p} \bE[|Y|^q | \sG]^{1/q} \leq \norm{\bE[|X|^p | \sG]}_{\linf}^{1/p} \bE[|Y|^q | \sG]^{1/q}. 
\end{align*}
It follows that 
\begin{align*}
    \norm{ \bE[X \cdot Y | \sG]}_{\lque} \leq \norm{\bE[|X|^p | \sG]}_{\linf}^{1/p} \norm{Y}_{\lque}, 
\end{align*}
i.e. $C \leq \norm{\bE[|X|^p | \sG]}_{\linf}^{1/p}$. A similar argument gives the same inequality in the case $p = 1$. 

Now we show that $C \geq \norm{\bE[|X|^p | \sG]}_{\linf}^{1/p}$. 
First we suppose that $p > 1$. Let $G \in \mathcal{G}$, and set $Y =  1_G |X|^{p/q - 1} X$. Notice that $|Y| = 1_G |X|^{p/q}$, so $|Y|^q = 1_G |X|^p$, so $Y \in L^q$. Next, notice that 
\begin{align*}
    X \cdot Y = 1_G |X|^{p/q - 1} X \cdot X = 1_G |X|^{p/q + 1} = 1_G |X|^{\frac{p + q}{q}} = 1_G |X|^{p}, 
\end{align*}
and so 
\begin{align*}
    \bE[ X \cdot Y | \sG]^q = 1_G \bE[|X|^p | \sG]^q. 
\end{align*}
So, we deduce from \eqref{assump} that 
\begin{align*}
    \bE[ 1_G \bE[|X|^p | \sG]^q] \leq C^q\bE[ |Y|^q] = C^q\bE[1_G |X|^p] = C^q\bE[1_G \bE[|X|^p | \sG]]
\end{align*}
holds for all $G \in \sG$. If $a > 0$ and $\bP[ \bE[|X|^p | \sG] > a]] > 0 $, then $\bE[|X|^p | \sG]^q \geq a^{q - 1} \bE[|X|^p | \sG ]$ on $G \coloneqq \{ \bE[|X|^p | \sG] > a\}$, so 
\begin{align*}
    a^{q-1} \bE[ 1_G \bE[|X|^p | \sG] ] \leq \bE[ 1_G\bE[|X|^p | \sG ]^q] \leq C^q \bE[ 1_G\bE[|X|^p | \sG]], 
\end{align*}
so $a^{q-1} \leq C^q$, i.e. $a \leq C^{\frac{q}{q-1}} = C^p$. This shows that $C \geq \norm{ \bE[|X|^p | \sG]}_{\linf}^{1/p}$ holds in the case $p > 1$.

Now suppose that $X \in \lone(\Omega; \R^n)$, and 
Define $Y$ by $Y = \frac{X}{|X|}1_{X \neq 0}$. Then we have $X \cdot Y = |X|$, and so by \eqref{assump},
    $\norm{\bE[|X| | \sG]}_{\linf} \leq C$, 
as desired.
\end{proof}

\begin{lemma}
\label{matrices}
Suppose that $A \in \lpee(\Omega ; \R^{n \times n})$, and
\begin{align} \label{assump2}
   \norm{\bE[ AY | \sG ]|}_{\lque} \leqc \norm{Y}_{\lque}
\end{align}
holds for all $Y \in \lque(\Omega; \R^n)$. Then there is a constant $C'$ depending only on $C$ such that
\begin{align} \label{result2}
    \norm{\bE[ |A|^p | ]}_{\linf} \leq C'. 
\end{align}
\end{lemma} 
\begin{proof}
Denote by $A^i$ the $i^{th}$ row of the matrix $A$. Notice that 
\begin{align*}
    |\bE[ AY | ] | \geq |\bE[A^i \cdot Y | ]|, 
\end{align*}
and so from \eqref{assump2} we get 
\begin{align*}
   \norm{\bE[ A^i \cdot Y | ]|}_{\lque} \leq_{C} \norm{Y}_{\lque}. 
\end{align*}
Applying Lemma \ref{vectors} gives $||\bE[ |A^i|^p | ] ||_{L^{\infty}} < C^{\frac{q}{q-1}}$, and since this holds for each $i$, we get the result. 
\end{proof}

\begin{lemma} \label{lem:martfromunique}
If $S$ satisfies $\ro$ and uniqueness in $\sinf$ holds for HBSDE($\kA$), then $S$ is a true martingale.
\end{lemma}

\begin{proof}
Let $e_j$ be the $j^{th}$ standard basis vector in $\R^n$, and set $\xi = e_j$. Then the unique solution to \eqref{lin-bsde-2} must satisfy $Y_t = \xi = e_j$. But by Proposition \ref{pro:linfexist}, we see that 
\begin{align*}
    e_j = Y_t = S_t^{-1}\bE_t[S_T \xi] = S_t^{-1}\bE_t[S_T]e_j,
\end{align*}
and so for each $t$, we have $\bE_t[S_T] = S_t$, which shows that $S$ is a true martingale. 
\end{proof}

\begin{proposition} \label{pro:rpfromeqn}
Suppose that for HBSDE($\kA$) is well-posed in $\sque$, for some $q \in (1, \infty]$. 
Then, $S$ is a true martingale that satisfies $\rp$, where $p$ is the conjugate of $q$. Moreover, $\rpc(S)$ depends only on the constant $C$ which verifies the estimate 
\begin{align*}
    \norm{Y}_{\sque} \leqc \norm{\xi}_{\lque}
\end{align*}
for each $\xi \in \sque$. 
\end{proposition}

\begin{proof}
Let $\tau_k = \inf \{t \geq 0 : \max \{|S_t|, |S_t^{-1}| \geq k\}\}$. Let $\sigma$ be an arbitrary stopping time. For each $k \in \N$ and $\xi \in \sF_{\tau_k}$, we have
\begin{align*}
    S_{\tau_k \wedge \sigma} Y_{\tau_k \wedge \sigma} = \bE_{\sigma}[S_{\tau_k} Y_{\tau_k}] =\bE_{\sigma}[S_{\tau_k}\xi],
\end{align*}
where we have used uniqueness and the fact that $\xi \in \sF_{\tau_k}$ to ensure that $Y_{\tau_k} = \xi$. From well-posedness we see that 
\begin{align*}
    \norm{\bE_{\sigma}[S_{\tau_k \wedge \sigma}^{-1} S_{\tau_k} \xi]}_{\lque} = \norm{Y_{\tau_k \wedge \sigma}}_{\lque} \leqc \norm{\xi}_{\lque}.
\end{align*}
It follows from Corollary \ref{matrices} that if $q = \infty$) that $\norm{ \bE_{\sigma}[|S_{\tau_k \wedge \sigma}^{-1} S_{\tau_k}|^p]}_{\linf} \leq C'$ for some $C'$ depending only on $C$. An application of the conditional Fatou's lemma shows that $S$ satisfies $\rp$. With this in hand, that $S$ is a true martingale follows from Lemma \ref{lem:martfromunique}.
\end{proof}

Combining Proposition \ref{pro:rpfromeqn} with the conditional H\"older's inequality gives the following corollary. 

\begin{corollary} \label{cor:wellposedsbiggerq}
Suppose that $1 < q < q' \leq \infty$, and BSDE($\kA$) is well-posed in $\sque$. Then BSDE($\kA$) is well-posed in $\mathcal{S}^{q'}$, and \begin{align*}
    \opnorm{S_{\kA}}_{q'} \leq C, \, \, \const{\opnorm{S_{\kA}}_q, \norm{\kA}_{\bmo}, q, q'}. 
\end{align*}
\end{corollary}

\begin{proof}
If BSDE($\kA$) is well-posed in $\sque$, then by Proposition \ref{pro:rpfromeqn} and Lemma \ref{lem:martfromunique}, $S$ is a true martingale satisfying $\rp$, and $\rpc(S) \leq C$, $\const{\opnorm{S_{\kA}}_q}$. By the conditional reverse H\"older inequality, if $p'$ is the conjugate of $q'$ and $p$ is the conjugate of $q$, we have
\begin{align*}
   \bE_{\tau}[|S_{\tau}^{-1}S_T|^{p'}] \leq \bE_{\tau}[|S_{\tau}^{-1}S_T|^{p}]^{p'/p},  
\end{align*}
and so $\text{R}_{p'}(S) \leq \rpc(S)^{p'/p}$. The result now follows form Propositions \ref{pro:lqexist} and \ref{pro:yrep} (or Propositions \ref{pro:linfexist} and \ref{pro: yrep1} if $q' = \infty$). 
\end{proof}

Propositions \ref{pro:rpfromeqn}, \ref{pro: yrep1}, \ref{pro:yrep}, \ref{pro:linfexist} and \ref{pro:lqexist} now combine to give the following equivalence.

\begin{theorem} \label{thm:main}
Let $q \in (1,\infty]$, and $p$ be the conjugate of $q$. Then the following are equivalent. 
\begin{enumerate}
    \item $S$ is a true martingale which satisfies $\rp$. 
    \item BSDE($\kA$) is well-posed in $\sque$. 
    \item HBSDE($\kA$) is well-posed in $\sque$. 
\end{enumerate}
\end{theorem}

\begin{proof}
(1) $\Implies$ (2): If $p = 1$, this implication is given directly by Propsitions \ref{pro: yrep1} and \ref{pro:linfexist}. If $p > 1$, then by Proposition \ref{pro:largerp} there eists $p' > p$ such that $S$ satisfies $(R_{p'})$, and then well-posedness of \eqref{lin-bsde-1} in $\sque$ follows from Propositions \ref{pro:yrep} and \ref{pro:lqexist}.
\newline \newline
(2) $\Implies$ (3) is obvious. 
\newline \newline
(3) $\Implies$ (1) is given by Proposition \ref{pro:rpfromeqn}. 
\end{proof}

Very similar arguments let us state the following regarding strong well-posedness in $\sinf$. 

\begin{proposition} \label{pro:linfexists}
Suppose that $S$ satisfies $\rp$ for some $p > 1$ and that $S$ is a true martingale. Then BSDE($\kA$) is strongly well-posed in $\sinf$. Moreover, $\opnorm{S_{\kA}} \leq C$, $\const{\rpc(S), \norm{\kA}_{\bmo}}$. 
\end{proposition}

\section{Structural conditions} 
\label{sec:structural}

In this section, we explain how the equivalent conditions appearing in Theorem \ref{thm:main} can be verified under various structural conditions on $\kA$. Several results in this sections will be straightforward adaptations of some results in \cite{jackson2021existence}, and so will be given without proof. We emphasize that, as evidenced by the proofs of Propositions \ref{pro:rop} and \ref{pro:lop}, it is sometimes easier to verify well-posedness directly (and get the reverse H\"older inequality as a corollary), and sometimes easier to verify the reverse H\"older inequality (and get well-posedness as a corollary).

Here we define the structural conditions which we study. 

\begin{definition}
Let $\kA$ be a matrix in $\bmo((\R^d)^{n \times n})$. We say that $\kA$
\begin{itemize}
    \item is \textbf{lower triangular} if $\kA^i_j = 0$ whenever $j > i$
    \item has \textbf{right outer-product structure} if there exist $\ka \in \bmo((\R^d)^n)$ and $b \in \R^n$ such that $\kA^i_j = \ka^i b_j$, and in this case we write $\kA = \ka b^T$
    \item has \textbf{left outer-product structure} if there exist $a \in \R^n$ and $\kb \in \bmo((\R^d)^n)$ such that $\kA^i_j = a_i \kb^j$, and in this case we write $\kA = a \kb^T$. 
\end{itemize} 
\end{definition}

We also recall the definition of sliceability. 

\begin{definition}
Let $\gamma \in \bmo$. We say that $\gamma$ is sliceable if for each $\epsilon > 0$, there is an $n \in \N$ sequence of stopping times $0 = \tau_0 \leq \tau_1 ... \leq \tau_n = T$ such that for each $i \in \{1,...,n\}$, $\norm{\gamma 1_{[\tau_{i-1}, \tau_i]}}_{\bmo} < \epsilon$. 
\end{definition}

Lower triangular matrices were studied in \cite{jackson2021existence}, and the following proposition is a slight adaptation of Proposition  2.6 of \cite{jackson2021existence}. 

\begin{proposition} \label{pro:triangular}
If $\kA$ is lower-triangular, then there are constants $q^* = q^*(\norm{\kA}_{\bmo})$, $C_q = C_q(\norm{\kA}_{\bmo})$ such that BSDE($\kA$) is well-posed in $\sque$ for each $q > q^*$, and moreover $\opnorm{S_A}_q \leq C_q$.
\end{proposition}

Next, we give analogous results for matrices with left or right outer-product structure.
\begin{proposition} \label{pro:rop}
If $\kA = \ka b^T$ has right outer-product structure, then there exists $q^* = q^*(\norm{\ka}_{\bmo}, |b|)$ and constants $C_q = C_q(\norm{\ka}_{\bmo}, |b|)$ such that BSDE($\kA$) is well-posed in $\sque$ for each $q > q^*$, and
    $\norm{S_A}_{q} \leq C_q$.

\end{proposition}

\begin{proof}
By Theorem \ref{thm:1dwellposed}, we may choose $q^* = q^*(\norm{\ka}_{\bmo}, |b|)$ such that the one-dimensional BSDE 
\begin{align} \label{1deqn}
    U_{\cdot} = \eta + \int_{\cdot}^T  (b^T \ka) \kV  dt - \int_{\cdot}^T \kV d \kB.
\end{align}
is well-posed in $\sques$. 
Given $\xi \in \lques$, define $(U, \kV) \in \sques \times \ltqs$ to be the unique solution to \eqref{1deqn} with $\eta = b^T \xi$. Next, let $(Y, \kZ)$ solve the system
\begin{align} \label{ydef}
    Y_{\cdot} = \xi + \int_{\cdot}^T \ka \kV  dt - \int_{\cdot}^T \kZ d \kB.
\end{align}
We now claim that $\kV = b^T \kZ$. Indeed, both $(U,\kV)$ and $(b^T Y, b^T \kZ)$ solve the ``martingale representation equation" 
\begin{align*}
    \tilde{Y}_{\cdot} = b^T \xi + \int_{\cdot}^T \big( (b^T \ka) \kV + b^T \beta) dt - \int_{\cdot}^T \tilde{\kZ} d \kB,
\end{align*}
for which uniqueness holds. It follows that $(Y,\kZ)$ solves \eqref{lin-bsde-1}. The well-posedness of \eqref{1deqn} provides an estimate on $\kV$ of the form $\norm{\kV}_{\ltqs} \leqc \norm{b^T \xi}_{\lques}$, $\const{\norm{b^T\ka}_{\bmo}}$ and the necessary estimates on $(Y,\kZ)$ follow from \eqref{ydef} through standard techniques. Uniqueness can similarly be obtained from uniqueness for the equations  \eqref{1deqn} and \eqref{ydef}. This shows that BSDE($\kA$) is well-posed in $\sques$. The full result follows from Corollary \ref{cor:wellposedsbiggerq}. 
\end{proof}

 For $\klambda \in (\R^d)^n$, we write $\text{Diag}(\klambda)$ to mean the process in $\bmo((\R^d)^{n \times n})$ with $\big(\text{Diag}(\klambda)\big)^i_j = \delta_{ij} \klambda$. The following slight extension of Proposition \ref{pro:rop} will be useful in the analysis of quadratic-linear drivers in Section \ref{sec:ql}. The proof is almost identical, so is omitted. 

\begin{proposition} \label{pro:outerdiag}
Suppose that $\kA \in \bmo((\R^d)^{n \times n})$ can be written in the form 
\begin{align*}
    \kA = \ka b^T + \text{Diag}(\klambda), 
\end{align*}
where $\ka \in \bmo(\R^n)$, $b \in \R^n$, and $\klambda \in \bmo(\R^d)$. Then there exists $q = q(\norm{a}_{\bmo}, |b|, \norm{\lambda}_{\bmo})$ such that BSDE($\kA$) is well-posed in $\sque$, and we have 
\begin{align*}
    \opnorm{S_A}_{q} \leq C, \, \, \const{\norm{a}_{\bmo}, |b|, \norm{\klambda}_{\bmo}}.  
\end{align*}
\end{proposition}

\begin{proposition} \label{pro:lop}
If $\kA = a \kb^T$ has left outer-product structure, then there exist constants $q^* = q^*(|a|, \norm{\kb}_{\bmo})$ and $C_q = C_q(|a|, \norm{\kb}_{\bmo})$ such that BSDE($\kA$) is well-posed in $\sque$ for each $q > q^*$, and $\opnorm{S_{\kA}}_q \leq C_q$.
\end{proposition}

\begin{proof}
Since $\kA = a \kb^T$, the equation for the stochastic exponential $S$ reads 
\begin{align} \label{lopeqn}
    \begin{cases}
    dS = Sa \kb^T d\kB, \\
    S_0 = I_{n \times n}. 
    \end{cases}
\end{align}
Therefore the process $Sa$ takes values in $\R^n$, and satisfies 
\begin{align*}
    \begin{cases} 
    d(Sa) = (Sa) (\kb^T a) d\kB, \\
    (Sa)_0 = a.
    \end{cases}
\end{align*}
That is, for $i \in \{1,...,n\}$, we have $(Sa)_i = a_i \mathcal{E}(\int (\kb^T a) d \kB)$. Since $\norm{\kb^T a}_{\bmo} \leq |a|\norm{\kb}_{\bmo}$, we have by Theorem \ref{thm:1drh} that there exists $p = p(|a|,\norm{\kb})$ such that 
\begin{align*}
    \norm{Sa}_{\spee} \leq C, \, \, \const{|a|, \norm{\kb}}. 
\end{align*}
It now follows from Lemma 1.4 of \cite{Delbaen-Tang} (and \eqref{lopeqn}) that $\norm{S}_{\spee} \leq C$, $\const{|a|, \norm{\kb}_{\bmo}}$. A conditioning argument (see the proof of Theorem 3.1 in \cite{Kazamaki} for an example) allows us to conclude that in fact $S$ satisfies the inequality 
\begin{align*}
\bE[\sup_{ \tau \leq t \leq T} |S_{\tau}^{-1}S_t|^p] \leq C, \, \, \const{|a|, \norm{\kb}_{\bmo}}, 
\end{align*}
and so $S$ is a true martingale satisfying $\rp$.
The result now follows from Theorem \ref{thm:main}. 
\end{proof}

Once again, it useful to note that we can add a diagonal matrix, and the same argument goes through. 

\begin{proposition} \label{pro:lopdiag}
If $\kA = a \kb^T + Diag(\lambda)$, where $a \in \R^n$, $\kb \in \bmo(\R^d)^n$, and $\lambda \in \bmo(\R^d)$, then there exist constants $q^* = q^*(|a|, \norm{\kb}_{\bmo}, \norm{\lambda}_{\bmo})$ and $C_q = C_q(|a|, \norm{\kb}_{\bmo})$ such that BSDE($\kA$) is well-posed in $\sque$ for each $q > q^*$, and $\opnorm{S_{\kA}}_q \leq C_q$.
\end{proposition}

The following proposition deals with sliceability, and is a slight extension of Theorem 2.9 of \cite{jackson2021existence}. It states, among other things, that our results about equation \eqref{lin-bsde-1} extend to the more general BSDE 
\begin{align} \label{lin-bsde-y}
    Y_{\cdot} = \xi + \int_{\cdot}^T \big(\alpha Y + \kA \kZ + \beta \big) dt - \int_{\cdot}^T \kZ d\kB. 
\end{align}
\begin{proposition} \label{pro:sliceable}
If BSDE($\kA$) is strongly well-posed in $\sinf$ and $\Delta \kA$ is sliceable, then BSDE($\kA + \Delta \kA$) is strongly well-posed in $\sinf$. Moreover, if $\Delta \kA \in \lii$, and $\alpha \in \lii$, then for each $(\xi, \beta) \in \linf \times \bmoh$, there exists a unique solution to the equation \eqref{lin-bsde-y} satisfying
\begin{align*}
    \norm{Y}_{\sinf} + \norm{\kZ}_{\bmo} \leqc \norm{\xi}_{\linf} + \norm{\beta}_{\bmoh}, \, \, \const{\opnorm{S_{\kA}}_{\infty,s}, \norm{\alpha}_{\lii}, \norm{\Delta \kA}_{\lii}}.
\end{align*}
If $1 < q < \infty$, BSDE($\kA$) is well-posed in $\sque$ and $\Delta \kA$ is sliceable, then BSDE($\kA + \Delta \kA$) is well-posed in $\sque$. Moreover, if $\Delta \kA \in \lii$, and $\alpha \in \lii$, then for each $(\xi, \beta) \in \lque \times \loq$, there exists a unique solution to the equation \eqref{lin-bsde-y} satisfying
\begin{align*}
   \norm{Y}_{\sque} + \norm{\kZ}_{\ltq} \leqc \norm{\xi}_{\linf} + \norm{\beta}_{\loq}, \, \, \const{\opnorm{S_{\kA}}_{q}, \norm{\alpha}_{\lii}, \norm{\Delta \kA}_{\lii}}.
\end{align*}
\end{proposition}

We now list some interesting conclusions which can be drawn by combining Theorem \ref{thm:main} with the results of this section. 

\begin{corollary}
Let $\kA_1, \kA_2 \in \bmo((\R^d)^{n \times n})$, and suppose that $\kA_1 - \kA_2$ is sliceable. Denote by $S_1$, and $S_2$ the stochastic exponentials corresponding to $\kA_1$, $\kA_2$, respectively. Then for any $1 < p < \infty$, $S_1$ is a true martingale satisfying $\rp$ if and only if $S_2$ is.
\end{corollary}

\begin{corollary}
 If $\kA$ is lower triangular, or has left or right outer-product structure, then $S$ is a true martingale which satisfies $\rp$ for some $p > 1$.
\end{corollary}

\section{Application to quadratic BSDE systems} \label{sec:qbsde}

\begin{definition}\label{def:driver} A \define{driver} is a random field $f : [0,T]
  \times \Omega \times \R^n \times (\R^d)^n \to \R^n$ such that
  \begin{enumerate}
    \item $f(\cdot,\cdot,y,\kz)$ is progressively measurable process for all
    $y,\kz$.
    \item $f(\cdot,\omega,\cdot,\cdot)$ is a continuous function for each
    $\omega$.
  \end{enumerate}
  \end{definition}
In this section, we consider the BSDE 
\begin{align} \label{bsde}
    Y_{\cdot} = \xi + \int_{\cdot}^T f(\cdot, Y,\kZ) dt - \int_{\cdot}^T \kZ d\kB
\end{align}
where $\xi \in \linf(\R^n)$ and $f$ is a driver satisfying certain structural conditions. If $f$ is $C^1$ (i.e., $(y,\kz) \mapsto f(t, \omega, y, \kz)$ is $C^1$ for each $(t, \omega) \in [0,T] \times \Omega$), then we write $\frac{\partial f}{ \partial y}$ and $\frac{\partial f}{\partial z}$ for the derivatives with respect to $y$ and $z$, respectively. We view $\pd{f}{y}$ and $\pd{f}{z}$ as maps $[0,T] \times \Omega \times \R^n \times (\R^d)^n \to \R^{n \times n}$ and $
 [0,T] \times \Omega \times \R^n \times (\R^d)^n \to (\R^d)^{n \times n}$, respectively. We will use the following definitions when discussing the drivers of interest. 

\begin{definition}
A driver $f$ is said to
\begin{itemize}
    \item be \define{Lipschitz} if there is a constant $L$ such that for all $y, y' \in \R^n$, $z,z' \in (\R^d)^n$, and $(t,\omega) \in [0,T] \times \Omega$, we have
    \begin{align*}
        |f(t,\omega, y, \kz) - f(t,\omega, y',\kz')| \leq L \big(|y - y'| + |z - z'|\big), 
    \end{align*}
    and $\norm{f(\cdot, \cdot, 0,0)}_{\lii} \leq L$. In this case we write $f \in \lip(L)$
    \item
    be \define{Malliavin-regular} if there exists a constant $L>0$ and a random field $D_{\cdot}f: [0,T]^2 \times
  \Omega \times \R^d \times (\R^d)^n \to \R$ such that
    \begin{enumerate}
      \item for all $(y,\kz)$, $D_{\cdot}f(\cdot,y,\kz)$ is a version of the
      Malliavin derivative of the process $f(\cdot, y, \kz)$,
      \item $\abs{Df}\leq L$, and
      \item $\abs{ D_{\cdot}f(\cdot, y', \kz') - D_{\cdot}f(\cdot, y, \kz)} 
      \leq  L \Big( \abs{y'-y} + \abs{\kz' - \kz} \Big)$.
    \end{enumerate}
  In this case, we write $f \in \mal(L)$.
  \item \define{satisfy the condition (AB)} if there is a
  process $\rho \in \loi$ and a finite collection $\sam = (a_1,\dots, a_M)$ of
  vectors in $\R^n$ such that
  \begin{enumerate}
    \item $a_1,\dots, a_M$ positively span $\R^n$ 
    \item $ a_m^T f(t,\omega,y, \kz) \leq \rho + \tfrac{1}{2} \abs{a_m^T \kz}^2$
    for each $m$, for all $y,\kz$.
  \end{enumerate}
  In this case, we say that $f \in \apb(\rho, \sam)$, with
  $\apb=\cup_{\rho,\sam} \apb(\rho,\sam)$.
\end{itemize}
\end{definition}

\begin{remark}
We remark that the condition (AB) does not appear as a hypothesis in Theorem \ref{thm:qlmain} regarding quadratic linear drivers, but it is a hypothesis in Theorem \ref{thm:udmain} which concerns unidirecitonal quadratic drivers.
\end{remark}

We will also need the notion of a Lyapunov function, as introduced in \cite{xing2018}.

\begin{definition} \label{def:lyapunov}
Let $f$ be a driver, and $c$ a constant. A non-negative function $h \in C^2(\R^n)$ is a \textbf{c-Lyapunov function} for $f$ if $h(0) = 0$, $Dh(0) = 0$, and 
\begin{align*}
    \frac{1}{2} \sum_{i,j = 1}^n (D^2h(y))_{ij} \kz^i \cdot \kz^j - Dh(y) \cdot f(t, \omega, y, z)  \geq |z|^2 - k
\end{align*}
for all $(t, \omega, y, \kz) \in [0,T] \times \Omega \times \R^n \times (\R^d)^n$ with $|y| \leq c$. In this case, we say that $(h,k) \in \textbf{Ly}(f,c)$.
\end{definition}

The utility of Lyapunov functions is demonstrated by the following Lemma.

\begin{lemma} \label{lem:zfromy}
Suppose that for some constant $c$ and driver $f$, there is a Lyapunov pair $(h,k) \in \mathbf{Ly}(f, c)$. Suppose further that there is a solution $(Y,\kZ) \in \sinf \times \bmo$ to \eqref{bsde} such that $\norm{Y}_{\sinf} \leq c$. Then we have the estimate 
\begin{align*}
    \norm{\kZ}_{\bmo}^2 \leq kT + 2 \sup_{|y| \leq c} |h(y)|. 
\end{align*}
\end{lemma}

\begin{proof}
By the definition of $\mathbf{Ly}(f,c)$, we find that 
\begin{align*}
    h(Y_{\cdot})  + \int_0^{\cdot} k - |\kZ|^2 dt
\end{align*}
is a submartingale, and so for any stopping time $\tau$ we have
\begin{align*}
    \bE_{\tau}[h(Y_T) +  \int_0^T k - |\kZ|^2 dt] \geq h(Y_{\tau})  + \int_0^{\tau} k - |\kZ|^2 dt, 
\end{align*}
and rearranging gives
\begin{align*}
    \bE_{\tau}[\int_{\tau}^T |\kZ|^2 dt] \leq k(T - \tau) + \bE_{\tau}[h(Y_T) - h(Y_{\tau})]. 
\end{align*}
The result follows. 
\end{proof}

The reason for the name of the a-priori boundedness condition is explained by the following lemma. 

\begin{lemma} \label{lem:ab}
Suppose that $f$ is a driver which satisfies the condition (AB), and that $(Y,\kZ) \in \sinf \times \bmo$ solves \eqref{bsde}. Then we have the estimate
\begin{align*}
    \norm{Y}_{\sinf} + \norm{\kZ}_{\bmo} \leq C, \, \, \const{\norm{\xi}_{\linf}, \rho, \{a_m\}}. 
\end{align*}
\end{lemma}

\begin{proof} 
The proof is almost identical to that of Proposition 3.8 of \cite{jackson2021existence}, and so is omitted.  
\end{proof}

Finally, we define the space of bounded random variables with bounded Malliavin derivatives as follows. 

\begin{definition}
If $\xi$ is Malliavin differentiable, we denote by $D \xi$ its Malliavin derivative. We denote by $\doi$ the space of all $\xi \in \linf$ such that $\xi$ is Malliavin differentiable and $D\xi \in \lii$. 
\end{definition}

\subsection{BSDEs with quadratic-linear drivers} \label{sec:ql}

\begin{definition}
  A driver $f$ is called  \define{quadratic-linear} if there is a constant $L$, a driver $g \in \lip(L) \cap \mal(L)$, and a vector $b \in \R^n$ with $|b| \leq L$ such that 
    \begin{align*}
        f(t, \omega, y ,\kz) = g(t,\omega, y, \kz) + \kz b^T \kz. 
    \end{align*}
    In this case, we write $f \in \ql(L)$, and set $\ql = \cup_L \ql(L)$. 
\end{definition}

Let us explain precisely what is meant by $\kz b^T \kz$. Here $\kz \in (\R^d)^n$, so $b^T \kz \in \R^d$ is given by 
\begin{align*}
    b^T \kz = \sum_{i = 1}^n b_i \kz^i, 
\end{align*}
and thus $\kz b^T \kz \in \R^n$ is given by 
\begin{align*}
    (\kz b^T \kz)_i = \kz^i \cdot (b^T \kz) = \kz^i \cdot \big( \sum_{j = 1}^n b_j \kz^j \big) = \sum_{j = 1}^n b^j \kz^i \cdot \kz^j. 
\end{align*}
Thus the quadratic term in a quadratic-linear driver is given by a certain weighted sum of inner products of the components of $\kz \in (\R^d)^n$.

Here is the main result regarding quadratic-linear drivers. 

\begin{theorem} \label{thm:qlmain}
Suppose that $\xi \in \linf$ and $f \in \ql(L)$. Then, there exists a unique solution $(Y,\kZ) \in \sinf \times \bmo$ to \eqref{bsde}, which satisfies the estimates 
\begin{align*}
    \norm{Y}_{\sinf} + \norm{\kZ}_{\bmo} \leq C, \text{ where } \const{L, \norm{\xi}_{\linf}}.
\end{align*}
If in addition $\xi \in \doi$, then we have 
\begin{align*}
    \norm{\kZ}_{\lii} \leq C, \text{ where } \const{L, \norm{\xi}_{\linf}, \norm{D\xi}_{\lii}}.
\end{align*}
\end{theorem}

\begin{remark}
It is possible via an approximation argument to remove the smoothness hypotheses on the Lipschitz driver $g$. It is also possible, if we assume that $f$ satisfies the condition (AB), to assume that that $g$ is merely sub-quadratic rather than Lipschitz. In order to avoid additional technicalities we will not pursue these generalizations. 
\end{remark}

\subsubsection{A-priori estimates for quadratic-linear drivers}

We start by producing $\sinf \times \bmo$ a-priori estimates for equations with quadratic-linear drivers. In fact, we need estimates for a slightly more general class of drivers.

One of the key results of \cite{xing2018} is the existence of Lyapunov functions when the driver satisfies the so-called Bensoussan-Frehse condition. We only need a very special case of their result.
The proof of Proposition 2.11 in \cite{xing2018} does not rely on the Markovian property of $f$, and thus an identical argument shows the following.  

\begin{proposition} \label{prop:lyapunovexist}
Suppose that $f$ can be written as 
\begin{align} \label{qldecomp}
    f(t,\omega, y, \kz) = g(t,\omega, y, \kz) + \kz \psi(\kz) 
\end{align}
for some $g \in \lip(L)$, and $\psi: (\R^d)^n \to \R^n$ which is $L$-Lipschitz. Then, for any constant $c$, there is a function $h$ and a constant $k$ such that $(h,k) \in \mathbf{Ly}(f, c)$. Moreover, $(h,k)$ can be chosen so that that $(h,k) \in \mathbf{Ly}(f',c)$ whenever $f'$ has a decomposition of the form \eqref{qldecomp}. 
\end{proposition}

\begin{proposition} \label{pro:apriori}
Suppose that $f$ can be written as 
\begin{align*}
    f(t,\omega, y, \kz) = g(t,\omega, y, \kz) + \kz \psi(\kz) 
\end{align*}
for some $l \in \lip(L)$, and $\psi: (\R^d)^n \to \R^n$ which is $L$-Lipschitz. Suppose further that $(Y,\kZ) \in \sinf \times \bmo$ solves \eqref{bsde}. Then we have 
\begin{align*}
    \norm{Y}_{\sinf} + \norm{\kZ}_{\bmo} \leq C, \, \, \const{\norm{\xi}_{\linf}, L}. 
\end{align*}
\end{proposition}

\begin{proof}
To estimate $Y$, we first change measure, noting that 
\begin{align*}
    Y_{\cdot} = \xi + \int_{\cdot}^T g(\cdot, Y, \kZ) ds - \int_{\cdot}^T \kZ d\kB' 
\end{align*}
where by Girsanov $\kB' = \kB - \int \psi(\kZ) dt $ is a Brownian motion under the probability measure $\bQ$, with $\frac{d\bQ}{d\bP} = \mathcal{E}(\int \psi(\kZ) d\kB)_T$. The estimate on $Y$ now follows form standard techniques for estimating solutions to Lipschitz BSDEs, e.g. studying the It\^o decomposition of $\exp{\lambda t} |Y_t|^2$ for $\lambda$ large enough. The corresponding estimate for $\kZ$ follows from Lemma \ref{lem:zfromy} and Proposition \ref{prop:lyapunovexist}. 
\end{proof}

\subsubsection{Existence with smooth terminal condition} \label{sub:ql}

We note that it will be helpful in what follows to use the notations for finite differences and derivatives introduced in \cite{jackson2021existence}, which allows us for example to write
\begin{align*}
    f(\cdot, Y^{(1)}, \kZ^{(1)}) - f(\cdot, Y^{(2)}, \kZ^{(2)}) = \frac{\Delta f}{\Delta Y} \Delta Y + \frac{\Delta f}{\Delta \kZ} \Delta \kZ
\end{align*}
whenever $f$ is a driver and $Y^{(i)}$ and $\kZ^{(i)}$ are processes with appropriate dimensions and $\Delta Y = Y^{(1)} - Y^{(2)}$, $\Delta \kZ = \kZ^{(1)} - \kZ^{(2)}$. Here $\frac{\Delta f}{\Delta Y}$ and $\frac{\Delta f}{\Delta Z}$ are processes taking values in $\R^{n \times n}$ and $(\R^d)^{n \times n}$, respectively. In the case $d = 1$, for example we would have
\begin{align*}
    \frac{\Delta f}{\Delta Y} = 1_{\Delta Y \neq 0}\frac{f(\cdot, Y^{(1)}, Z^{(1)}) - f(\cdot, Y^{(2)}, Z^{(1)})}{|\Delta Y|^2} (\Delta Y)^T, \\
    \frac{\Delta f}{\Delta Z} = 1_{\Delta \kZ \neq 0}\frac{f(\cdot, Y^{(1)}, \kZ^{(2)}) - f(\cdot, Y^{(2)}, Z^{(2)})}{|\Delta Z|^2} (\Delta Z)^T. 
\end{align*}
The extension to $d > 1$ is natural, as is the similar notation which will be used for differentiating $f$ in the variables $y$ and $\kz$. 

The goal of this sub-section is to prove the following Proposition.

\begin{proposition} \label{pro:smoothexist}
Suppose that $\xi \in \doi$ and $f \in \ql(L)$. Then, there exists a unique solution $(Y,\kZ) \in \sinf \times \bmo$ to \eqref{bsde}, which satisfies the estimates 
\begin{align*}
    \norm{Y}_{\sinf} + \norm{\kZ}_{\bmo} \leq C, \, \, \const{L, \norm{\xi}_{\linf}}, \\
    \norm{\kZ}_{\lii} \leq C, \text{ where } \const{L, \norm{\xi}_{\linf}, \norm{D\xi}_{\lii}}.
\end{align*}
\end{proposition}

\begin{proof}
Since $f \in \ql(L)$, we may write $f(t,\omega, y, \kz) = l(t,\omega, y, \kz) + \kz b^T \kz$ for some $g \in \lip(L)$, $b \in \R^n$. We begin by selecting a sequence $\phi_k$ of maps $\phi_k : \R^d \to \R^d$ with the following properties. 
\begin{itemize}
    \item $\phi_k$ is smooth, radial, and $1$-Lipschitz, and compactly supported. 
    \item $\phi_k$ is equal to the identity on $B_k(0) \subset \R^d$. 
\end{itemize}
We then define a sequence of drivers $f^k$ by 
\begin{align*}
    f^k(t,\omega, y, \kz) = g(t,\omega, y, \kz) + \kz \phi^k(b^T \kz). 
\end{align*}
Then for each $k$, $f^k$ is Lipschitz in $y$ and $\kz$, and so standard theory for Lipschitz BSDEs gives us a solution $(Y^{(k)},\kZ^{(k)}) \in \sinf \times \bmo$ to the equation 
\begin{align} \label{ykeqn}
    Y_{\cdot}^{(k)} = \xi + \int_{\cdot}^T f^k(\cdot, Y^{(k)}, \kZ^{(k)}) dt - \int_{\cdot}^T \kZ^{(k)} d\kB. 
\end{align}
Since each $f^k$ lies in $\overline{\ql}( \max \{L, 1\})$, Proposition \ref{pro:apriori} provides a constant $C = C(\norm{\xi}_{\linf}, L)$ such that 
\begin{align*}
    \sup_k \big(\norm{Y^{(k)}}_{\sinf} + \norm{\kZ}^{(k)}_{\bmo}\big) \leq C. 
\end{align*}
Now for $\theta \leq t$, we take the Malliavin derivative of \eqref{ykeqn} (as justified by Proposition 5.3 of \cite{karoui}) and apply the chain rule and product rule to obtain 
\begin{align} \label{deriveqn}
    D_{\theta} Y^{(k)}_t = D_{\theta} \xi + \int_t^T \big(D_{\theta}\big( l(\cdot, Y^{(k)}, \kZ^{(k)})\big)_s + D_{\theta}\big(\kZ^{(k)} \phi^k(b^T \kZ^{(k)})\big)_s\big) ds - \int_t^T D_{\theta} \kZ^{(k)}_s d\kB_s  \nonumber \\
    = D_{\theta} \xi + \int_t^T \big(\alpha^k_s D_{\theta}Y^{(k)}_s + (\kA^k_s + \Delta \kA^k_s) D_{\theta} \kZ^{(k)}_s + \beta^k_s \big)dt - \int_t^T D_{\theta}\kZ_s^{(k)} d\kB_s, 
\end{align}
where 
\begin{align*}
    \alpha^k = \pd{l}{y}(\cdot, Y^{(k)}, \kZ^{(k)}), \, \,
    \beta^k = Dl^k(\cdot, Y^{(k)}, \kZ^{(k)}), 
    \\ \kA^k = \ka_k b^T + \text{Diag}(\phi^k(b^T\kZ^{(k)})), \, \,
    \Delta \kA^k = \pd{l}{\kz}(\cdot, Y^{(k)}, \kZ^{(k)}),
\end{align*}
and $\ka_k \in \bmo((\R^d)^n)$ is given by 
\begin{align*}
    (\ka_k)^i = \nabla\phi_k(\kZ^{(k)})^T (\kZ^{(k)})^i.
\end{align*}
We view this a as a BSDE on $[\theta, T]$. 
Notice that since $l \in \lip(L)$, $\phi^k$ is 1-Lipschitz, and $\sup_k \big(\norm{Y^{(k)}}_{\sinf} + \norm{\kZ^{(k)}}_{\bmo}\big) \leq C, \const{L, \norm{\xi}_{\linf}}$, we must have 
\begin{align*}
    \norm{\ka_k}_{\bmo} + \norm{\alpha^k}_{\lii} + \norm{\beta^k}_{\lii} + \norm{\phi^k(b^T \kZ^{(k)})}_{\bmo} \leq C, \const{\norm{\xi}_{\lii}, L}. 
\end{align*}
Propositions \ref{pro:outerdiag} and \ref{pro:linfexists} thus show that that BSDE($\kA$) is strongly well-posed in $\sinf$, with $\opnorm{S_{\kA}}_{\infty, s} \leq C$, $\const{\norm{\xi}_{\linf}, L}$. Then Proposition \ref{pro:sliceable}, together with the fact that $D_\theta Y_{\theta} = Z_{\theta}$ (see \cite{karoui} for a precise statement), shows that
\begin{align*}
    \sup_k \norm{\kZ^{(k)}}_{\lii} \leq C, \, \, \const{L, \norm{\xi}_{\lii},\norm{D \xi}_{\lii}}. 
\end{align*}
Since $f^k(t,\omega, y, \kz) = f(t,\omega, y, \kz)$ for $|\kz| \leq k$, we conclude that $(Y^{(k)}, \kZ^{(k)})$ solves \eqref{bsde} as soon as $k > C$. 
\end{proof}

\subsubsection{Proof of Theorem \ref{thm:qlmain}}

We start with two estimates, which show a sort of stability (and in particular uniqueness) for \eqref{bsde}. 

\begin{proposition} \label{pro:sinfstab}
Suppose that $\xi \in \doi$ and $f \in \ql(L)$. Suppose further that $(Y, \kZ)$, $(Y',\kZ') \in \sinf \times \bmo$ solve \eqref{bsde} with terminal conditions $\xi$ and $\xi'$, respectively. Then we have 
\begin{align*}
    \norm{\Delta Y}_{\sinf} + \norm{\Delta \kZ}_{\bmo} \leqc \norm{\Delta \xi}_{\linf}, \, \, \const{L, \max \{\norm{\kZ}_{\bmo}, \norm{\kZ'}_{\bmo}\}}, 
\end{align*}
where $\Delta Y = Y - Y'$, $\Delta \kZ = \kZ - \kZ'$, and $\Delta \xi = \xi - \xi'$. 
\end{proposition}

\begin{proof}
We write 
\begin{align*}
    \Delta Y_{\cdot} = \Delta \xi + \int_{\cdot}^T \big(g(\cdot, Y^{(1)}, \kZ^{(1)}) - g(\cdot, Y^{(1)}, \kZ^{(1)}) + \kZ^{(1)}b^T \kZ^{(1)} - \kZ^{(2)} b^T \kZ^{(2)} \big)dt \\ - \int_{\cdot}^T \Delta \kZ d\kB \\
    = \Delta \xi + \int_{\cdot}^T \big(g(\cdot, Y^{(1)}, \kZ^{(1)}) - g(\cdot, Y^{(1)}, \kZ^{(1)})  + \Delta \kZ b^T \kZ^1 + \kZ^2 b^T \Delta \kZ\big) dt - \int_{\cdot}^T \kZ d\kB,
   \\ = \Delta \xi + \int_{\cdot}^T \big(\alpha \Delta Y + (\kA + \Delta \kA) \Delta \kZ \big) dt - \int_{\cdot}^T \kZ d\kB, 
\end{align*}
where 
\begin{align*}
    \alpha = \frac{\Delta l}{\Delta Y}, \, \, 
    \Delta \kA = \frac{\Delta l}{\Delta \kZ}, \, \,
    \kA = \kZ^2 b^T + \text{Diag}(b^T \kZ^1).
\end{align*}
The result now follows as in the proof of Proposition \ref{pro:smoothexist}, by applying Proposition \ref{pro:outerdiag} to estimate $\opnorm{S_{\kA}}_{\infty,s}$, and then Proposition \ref{pro:sliceable} to obtain the desired estimate. 
\end{proof}

\begin{corollary} \label{cor:uniqueness}
If $f \in \ql(L)$, then there is at most one solution $(Y, \kZ) \in \sinf \times \bmo$ to \eqref{bsde}. 
\end{corollary}

The following Proposition is established using the same linearization technique as Proposition \ref{pro:sinfstab}. 

\begin{proposition} \label{pro:sqstab}
Suppose that $\xi \in \doi$ and $f \in \ql(L)$. Suppose further that $\xi, \xi' \in \linf$, and that $(Y, \kZ)$, $(Y',\kZ') \in \sinf \times \bmo$ solve \eqref{bsde} with terminal conditions $\xi$ and $\xi'$, respectively. Then there are constants $q^* = q^*(\max \{\norm{\kZ}_{\bmo}, \norm{\kZ'}_{\bmo}\})$ and $C_q = C_q(\max \{\norm{\kZ}_{\bmo}, \norm{\kZ'}_{\bmo}\})$ such that, for each $q > q^*$, we have
\begin{align*}
    \norm{\Delta Y}_{\sque} + \norm{\Delta \kZ}_{\ltq} \leq_{C_q} \norm{\Delta \xi}_{\lque}.
\end{align*}
\end{proposition}

Now we complete the proof of Theorem \ref{thm:qlmain}. 

\begin{proof}
Uniqueness has been established by Corollary \ref{cor:uniqueness}. For existence, we choose a sequence $\xi^k \in \doi$ such that $\sup_k \norm{\xi}_{\linf} \leq \norm{\xi}_{\linf}$ and $\xi^k \xrightarrow{\lque} \xi$ for each $q > 1$. Let $(Y^{(k)},\kZ^{(k)}) \in \sinf \times \bmo$ be the unique solution (whose existence is guaranteed by Proposition \ref{pro:smoothexist}) to the equation
\begin{align*}
    Y^{(k)}_{\cdot} = \xi^k + \int_{\cdot}^T f(\cdot, Y^{(k)}, \kZ^{(k)}) dt - \int_{\cdot}^T \kZ^{(k)} d\kB. 
\end{align*}
Since $\sup_k \norm{\xi_k}_{\linf} \leq \norm{\xi}_{\linf}$, we have
\begin{align} \label{seqest}
    \sup_k \big(\norm{Y^{(k)}}_{\sinf} + \norm{\kZ^{(k)}}_{\bmo}\big) \leq C, \,\, \const{L, \norm{\xi}_{\linf}}. 
\end{align} 
Thus by Proposition \ref{pro:sqstab}, for some $q$ we have 
\begin{align*}
    \norm{Y^{(k)} - Y^{(j)}}_{\sque} + \norm{\kZ^{(k)} - \kZ^{(j)}}_{\ltq} \leqc \norm{\xi^k - \xi^j}_{\lque}, \, \, \const{L, \norm{\xi}_{\linf}}. 
\end{align*}
In particular, $\{(Y^{(k)}, \kZ^{(k)})\}$ is Cauchy in $\sque \times \ltq$, and so converges in $\sque \times \ltq$ to the desired solution $(Y,\kZ)$ to \eqref{bsde}. It follows from \eqref{seqest} that in fact $Y \in \sinf$ and $\norm{Y}_{\sinf} \leq \sup_k \norm{Y^{(k)}}_{\sinf} \leq C, \,\, \const{L, \norm{\xi}_{\linf}}$. To estimate $\norm{\kZ}_{\bmo}$, we note that we may assume without loss of generality that $q \geq 2$, so that 
\begin{align*}
    \bE[\int_0^T |\kZ^{(k)} - \kZ|^2 dt] = \norm{\kZ^{(k)} - \kZ}_{\ltt}^2 \leq \norm{\kZ^{(k)} - \kZ}_{\ltq}^2 \to 0.
\end{align*}
In particular, by passing to a subsequence if necessary, we may assume that $\int_0^T |\kZ^{(k)} - \kZ|^2 dt \to 0$ a.s., and so for any topping time $\tau$,
\begin{align*}
    \bE_{\tau}[\int_{\tau}^T |\kZ|^2 dt] = \bE_{\tau}[\lim_{k \to \infty} \int_{\tau}^T|\kZ^{(k)}|^2 dt] \leq \liminf_{k \to \infty} \bE_{\tau}[\int_{\tau}^T |\kZ^{(k)}|^2 dt].
\end{align*}
That $\kZ \in \bmo$, as well as the desired estimate on $\norm{\kZ}_{\bmo}$, now follows from \eqref{seqest}.
\end{proof}

\subsection{BSDEs with unidirectional quadratic drivers}

\begin{definition}
  A driver $f$ is called  \define{unidirectional quadratic} if there is a constant $L$, a driver $g \in \lip(L) \cap \mal(L)$, a constant $a \in \R^n$ and a $C^1$ function $h : (\R^d)^n \to \R$ such that
    \begin{enumerate}
        \item $f(t, \omega, y ,\kz) = g(t,\omega, y, \kz) + a h(\kz)$, 
        \item $|h(\mathbf{0})| \leq L$,
        \item $|h(\kz) - h(\kz')| \leq L(1 + |\kz| + |\kz'|)|\kz - \kz'|$.
    \end{enumerate}
    In this case, we write $f \in \ud(L)$, and set $\ud = \cup_L \ud(L)$. 
\end{definition}
\begin{remark}
To be clear, condition 1 above states that 
\begin{align*}
    f^i(t,\omega, y, \kz) = g^i(t,\omega, y, \kz) + a_i h(\kz), 
\end{align*}
where $a = (a_1,...,a_n)$. 
We choose the term unidirectional to reflect the fact that all of the quadratic growth in $f$ ``points" in the direction of $a \in \R^n$. 
\end{remark}
Here is the main result regarding unidirectional quadratic drivers. 

\begin{theorem} \label{thm:udmain}
Suppose that $\xi \in \linf$ and $f \in \ud(L) \cap \apb(\rho, \{a_m\})$. Then, there exists a unique solution $(Y,\kZ) \in \sinf \times \bmo$ to \eqref{bsde}, which satisfies the estimates 
\begin{align*}
    \norm{Y}_{\sinf} + \norm{\kZ}_{\bmo} \leq C, \text{ where } \const{ \norm{\xi}_{\linf}, \rho, \{a_m\}}.
\end{align*}
If in addition $\xi \in \doi$, then we have 
\begin{align*}
    \norm{\kZ}_{\lii} \leq C, \text{ where } \const{L, \norm{\xi}_{\linf}, \norm{D\xi}_{\lii}, \rho, \{a_m\}}.
\end{align*}
\end{theorem}

\begin{remark}
As with Theorem \ref{thm:qlmain}, it is possible in Theorem \ref{thm:udmain} to remove the smoothness assumptions on $g$ through an approximation argument, and it is also possible to assume that $g$ is merely sub-quadratic, rather than Lipschitz. 
\end{remark}

The rest of this subsection is devoted to a proof of Theorem \ref{thm:udmain}. We note that most of the arguments are very similar to the arguments in subsection \ref{sub:ql} (with Proposition \ref{pro:lop} replacing \ref{pro:outerdiag} and Lemma \ref{lem:ab} replacing Proposition \ref{pro:apriori}), and thus will be abbreviated.

\subsubsection{Existence with smooth terminal condition}

\begin{proposition} \label{pro:smoothexistud}
Suppose that $\xi \in \doi$ and $f \in \ud(L) \cap \apb(\rho, \{a_m\})$. Then, there exists a unique solution $(Y,\kZ) \in \sinf \times \bmo$ to \eqref{bsde}, which satisfies the estimates 
\begin{align*}
    \norm{Y}_{\sinf} + \norm{\kZ}_{\bmo} \leq C, \, \, \const{ \norm{\xi}_{\linf}, \rho, \{a_m\}}, \\
    \norm{\kZ}_{\lii} \leq C, \, \, \const{L, \norm{\xi}_{\linf}, \norm{D\xi}_{\lii}, \rho, \{a_m\}}.
\end{align*}
\end{proposition}

\begin{proof}
Let $g, h$, and $a$ be as given in the definition of unidirectional quadratic driver. We begin by selecting a sequence $\phi_k$ of maps $\phi_k : (\R^d)^n \to (\R^d)^n$ with the following properties. 
\begin{itemize}
    \item $\phi_k$ is smooth, radial, and $1$-Lipschitz, bounded, and such that $|\phi_k(\kz)| \leq |\kz|$. 
    \item $\phi_k$ is equal to the identity on $B_k(0) \subset (\R^d)^n$. 
\end{itemize}
We then define a sequence of drivers $f^k$ by 
\begin{align*}
    f^k(t,\omega, y, \kz) = g(t,\omega, y, \phi_k(\kz)) + a h(\phi_k(\kz)). 
\end{align*}
Then for each $k$, $f^k$ is Lipschitz in $y$ and $\kz$, and so standard theory for Lipschitz BSDEs gives us a solution $(Y,\kZ) \in \sinf \times \bmo$ to the equation 
\begin{align} \label{ykeqnud}
    Y_{\cdot}^k = \xi + \int_{\cdot}^T f^k(\cdot, Y^{(k)}, \kZ^{(k)}) dt - \int_{\cdot}^T \kZ^{(k)} d\kB. 
\end{align}
The properties of $\phi_k$ and the fact that $f \in \apb(\rho, \{a_m\})$ show that each $f^k \in \apb(\rho, \{a_m\})$, as well. Indeed, we simply check that 
\begin{align*}
    a_m^T f^k(t,\omega,y,\kz) = a_m^T f(t,\omega, y, \phi^k(\kz)) \leq \rho + \frac{1}{2} |a_m^T \phi^k(\kz)|^2 \leq \frac{1}{2}|a_m^T \kz|^2
\end{align*}
where the last inequality comes from the fact that $\phi_k$ is radial and $|\phi^k(\kz)| \leq |\kz|$. Thus Lemma \ref{lem:ab} gives the existence of a constant $\const{\norm{\xi}_{\linf}, \rho, \{a_m\}}$ such that
\begin{align*}
    \sup_k \big(\norm{Y^{(k)}}_{\sinf} + \norm{\kZ}^k_{\bmo}\big) \leq C. 
\end{align*}
Now for $\theta \leq t$, we take the Malliavin derivative of \eqref{ykeqn} (as justified by Proposition 5.3 of \cite{karoui}). We get
\begin{align} \label{deriveqnud}
D_{\theta} Y_t =
    D_{\theta} \xi + \int_t^T \big(\alpha^k_s Y_s + (\kA^k_s + \Delta \kA^k_s) D_{\theta} \kZ^{(k)}_s + \beta^k_s \big)dt - \int_t^T D_{\theta}Z_s^k d\kB_s, 
\end{align}
where 
\begin{align*}
    \alpha^k = \pd{g}{y}(\cdot, Y^{(k)},\kZ^{(k)}), \, \,
    \beta^k = Dg^k(\cdot, Y^{(k)}, \kZ^{(k)}),
    \\ \kA^k = a \kb_k^T, \, \,
    \Delta \kA^k = \pd{g}{\kz}(\cdot, Y^{(k)}, \kZ^{(k)}),  
\end{align*}
where $\kb^k = \nabla g(\kZ^{(k)})$, and $\nabla g : (\R^d)^n \to (\R^d)^n$ is the gradient of $g : (\R^d)^n \to \R$.
We view this a as a BSDE on $[\theta, T]$. Notice that since $g \in \lip(L)$, $\phi^k$ is $1$-Lipschitz, and $\sup_k \big(\norm{Y^{(k)}}_{\sinf} + \norm{\kZ^{(k)}}_{\bmo}\big) \leq C, \, \, \const{\norm{\xi}_{\linf}, \rho, \{a_m\}}$, and thus 
\begin{align*}
    \norm{\kb_k}_{\bmo} + \norm{\alpha^k}_{\lii} + \norm{\beta^k}_{\lii} \leq C, \, \, \const{\norm{\xi}_{\linf}, L, \rho, \{a_m\}}. 
\end{align*}
Propositions \ref{pro:lop} and \ref{pro:sliceable} thus show that that BSDE($\kA$) is strongly well-posed in $\sinf$, with $\opnorm{S_{\kA}}_{\infty, s} \leq C$, $\const{\norm{\xi}_{\linf}, L, \rho, \{a_m\}}$. The rest of the argument proceeds as in the proof of \ref{pro:smoothexist}.
\end{proof}

\subsubsection{Proof of Theorem \ref{thm:udmain}}

\begin{proposition} \label{pro:sinfstabud}
Suppose that $f \in \ud(L)$ and $\xi, \xi' \in \linf$. Suppose further that $(Y, \kZ)$, $(Y',\kZ') \in \sinf \times \bmo$ solve \eqref{bsde} with terminal conditions $\xi$ and $\xi'$, respectively. Then we have 
\begin{align*}
    \norm{\Delta Y}_{\sinf} + \norm{\Delta \kZ}_{\bmo} \leqc \norm{\Delta \xi}_{\linf}, \, \, \const{L, \max \{\norm{\kZ}_{\bmo}, \norm{\kZ'}_{\bmo}\}}, 
\end{align*}
where $\Delta Y = Y - Y'$, $\Delta \kZ = \kZ - \kZ'$, and $\Delta \xi = \xi - \xi'$. 
\end{proposition}

\begin{proof}
We write 
\begin{align*}
    \Delta Y_{\cdot} = \Delta \xi + \int_{\cdot}^T \big(g(\cdot, Y^{(1)}, \kZ^{(1)}) - g(\cdot, Y^{(1)}, \kZ^{(1)}) + a\big(h(\kZ^1) - h(\kZ^2)\big) - \int_{\cdot}^T \Delta \kZ d\kB 
   \\ = \Delta \xi + \int_{\cdot}^T \big(\alpha \Delta Y + (\kA + \Delta \kA) \Delta \kZ \big) dt - \int_{\cdot}^T \kZ d\kB, 
\end{align*}
where 
\begin{align*}
    \alpha = \frac{\Delta g}{\Delta Y}, \, \, 
    \Delta \kA = \frac{\Delta g}{\Delta \kZ}, \, \,
    \kA = a \kb^T, \, \, \kb = \frac{\Delta h}{\Delta \kZ}.
\end{align*}
The rest of the proof is the same as that of Proposition \ref{pro:sinfstab}, with Proposition \ref{pro:lop} taking the place of Proposition \ref{pro:outerdiag}. 
\end{proof}

\begin{corollary} \label{cor:uniquenessud}
If $f \in \ud$ and $\xi \in \linf$, then there is at most one solution $(Y, \kZ) \in \sinf \times \bmo$ to \eqref{bsde}. 
\end{corollary}

\begin{proposition} \label{pro:sqstabud}
Suppose that $f \in \ud(L)$ and $\xi, \xi' \in \linf$. Suppose further that $(Y, \kZ)$, $(Y',\kZ') \in \sinf \times \bmo$ solve \eqref{bsde} with terminal conditions $\xi$ and $\xi'$, respectively. Then there are constants $q^* = q^*(\max \{\norm{\kZ}_{\bmo}, \norm{\kZ'}_{\bmo}\})$ and $C_q = C_q(\max \{\norm{\kZ}_{\bmo}, \norm{\kZ'}_{\bmo}\})$ such that, for each $q > q^*$, we have
\begin{align*}
    \norm{\Delta Y}_{\sque} + \norm{\Delta \kZ}_{\ltq} \leq_{C_q} \norm{\Delta \xi}_{\lque}.
\end{align*}
\end{proposition}

\begin{proof}
Once again, the proof is the same as in the quadratic-linear case (Proposition \ref{pro:sqstab}) with Proposition \ref{pro:lop} replacing Proposition \ref{pro:outerdiag}. We omit the details.
\end{proof}

Now we complete the proof of Theorem \ref{thm:udmain}. 

\begin{proof}
Uniqueness follows from Corollary \ref{cor:uniquenessud}. For existence, we choose a sequence $\xi^k \in \doi$ such that $\sup_k \norm{\xi}_{\linf} \leq \norm{\xi}_{\linf}$ and $\xi^k \xrightarrow{\lque} \xi$ for each $q > 1$. Let $(Y^{(k)},\kZ^{(k)}) \in \sinf \times \bmo$ be the unique solution (whose existence is guaranteed by Proposition \ref{pro:smoothexist}) to the equation
\begin{align*}
    Y^{(k)}_{\cdot} = \xi^k + \int_{\cdot}^T f(\cdot, Y^{(k)}, \kZ^{(k)}) dt - \int_{\cdot}^T \kZ^{(k)} d\kB. 
\end{align*}
Lemma \ref{lem:ab} shows that since $\sup_k \norm{\xi_k}_{\linf} \leq \norm{\xi}_{\linf}$, we have
\begin{align*}
    \sup_k \big(\norm{Y^{(k)}}_{\sinf} + \norm{\kZ^{(k)}}_{\bmo}\big) \leq C, \,\, \const{\norm{\xi}_{\linf}, \rho, \{a_m\}}.
\end{align*} 
The remainder of the proof is the same as that of Theorem \ref{thm:qlmain}, with Proposition \ref{pro:sqstabud} playing the role of Proposition \ref{pro:sqstab}.
\end{proof}

\bibliographystyle{amsalpha}
\bibliography{reverseholder}
\end{document}